\RequirePackage[l2tabu,orthodox,abort]{nag}

\documentclass{birkjour}

\usepackage{fixltx2e}
\usepackage[all,error]{onlyamsmath}

\usepackage{amssymb,amsfonts,amsmath,
amsrefs,enumerate,color,url,hyperref,
mathbbol,dsfont}
\usepackage{mathrsfs}
\usepackage{tikz}

\usepackage[T1]{fontenc}
\usepackage[cp1250]{inputenc}

\usepackage[strict=true]{csquotes}

 \newtheorem{thm}{Theorem}[section]
 
 \newtheorem{lem}[thm]{Lemma}
 \newtheorem{prop}[thm]{Proposition}
 \theoremstyle{definition}
 
 \theoremstyle{remark}
 \newtheorem{rem}[thm]{Remark}
 
 \numberwithin{equation}{section}
\usepackage{mathtools}

\usepackage{enumitem}
\renewcommand{\theenumi}{\textup{(\alph{enumi})}}
\renewcommand{\labelenumi}{\theenumi}

\begin{document}

			
\title[Diffusion in thin layers separated by a semi-permeable membrane]{Modeling diffusion in thin $2D$-layers separated by a semi-permeable membrane}


\author[A. Bobrowski]{Adam Bobrowski}

\address{
Lublin University of Technology\\
Nadbystrzycka 38A\\
20-618 Lublin, Poland}

\email{a.bobrowski@pollub.pl}




\newcommand{\cxi}{(\xi_i)_{i\in \N} }
\newcommand{\lam}{\lambda}
\newcommand{\eps}{\varepsilon}
\newcommand{\ud}{\, \mathrm{d}}
\newcommand{\pr}{\mathbb{P}}
\newcommand{\f}{\mathcal{F}}
\newcommand{\s}{\mathcal{S}}
\newcommand{\h}{\mathcal{H}}
\newcommand{\ai}{\mathcal{I}}
\newcommand{\R}{\mathbb{R}}
\newcommand{\C}{\mathbb{C}}
\newcommand{\Z}{\mathbb{Z}}
\newcommand{\N}{\mathbb{N}}
\newcommand{\Y}{\mathbb{Y}}
\newcommand{\e}{\mathrm {e}}
\newcommand{\tif}{\widetilde {f}}
\newcommand{\slam}{\sqrt {\lam}}
\newcommand{\Id}{{\mathrm{Id}}}
\newcommand{\cic}{C_{\mathrm{mp}}}
\newcommand{\cod}{C_{\mathrm{odd}}[0,1]}
\newcommand{\cev}{C_{\mathrm{even}}[0,1]}
\newcommand{\cevr}{C_{\mathrm{even}}(\mathbb{R})}
\newcommand{\codr}{C_{\mathrm{odd}}(\mathbb{R})}
\newcommand{\cez}{C_0(0,1]}
\newcommand{\fod}{f_{\mathrm{odd}}} 
\newcommand{\fev}{f_{\mathrm{even}}} 
\newcommand{\sem}[1]{\mbox{$\left (\e^{t{#1}}\right )_{t \ge 0}$}}
\newcommand{\semi}[1]{\mbox{$\left ({#1}\right )_{t > 0}$}}
\newcommand{\semt}[2]{\mbox{$\left (\e^{t{#1}} \otimes_\varepsilon \e^{t{#2}} \right )_{t \ge 0}$}}
\newcommand{\tr}{\textcolor{red}}
\newcommand{\cea}{C_A}
\newcommand{\ceat}{C_A(t)}
\newcommand{\cosinea}{(\ceat )_{t\in \R}}  
\newcommand{\sea}{S_A}
\newcommand{\seat}{S_A(t)}
\newcommand{\sema}{(\seat )_{t\ge 0}}
\newcommand{\wt}{\widetilde}
\renewcommand{\iff}{if and only if }
\renewcommand{\k}{\mathrm{k}}
\newcommand{\tcm}{\textcolor{magenta}}
\newcommand{\tcb}{\textcolor{blue}}
\newcommand{\dx}{\ \textrm {d} x}
\newcommand{\dy}{\ \textrm {d} y}
\newcommand{\dz}{\ \textrm {d} z}
\newcommand{\di}{\textrm{d}}
\newcommand{\tcg}{\textcolor{green}}
\newcommand{\lc}{\mathfrak L_c}
\newcommand{\ls}{\mathfrak L_s}
\newcommand{\grat}{\lim_{t\to \infty}}
\newcommand{\grar}{\lim_{r\to 1-}}
\newcommand{\graR}{\lim_{R\to 1+}}
\newcommand{\grak}{\lim_{\kappa \to \infty}}
\newcommand{\gra}{\lim_{n\to \infty}}
\newcommand{\grae}{\lim_{\eps \to 0}}
\newcommand{\rez}[1]{\left (\lam - #1\right)^{-1}}
\newcommand{\papa}{\hfill $\square$}
\newcommand{\papap}{\end{proof}}
\newcommand {\x}{\mathbb{X}}
\newcommand{\aex}{A_{\mathrm ex}}
\newcommand{\jcg}[1]{\left ( #1 \right )_{n\ge 1} }
\newcommand {\y}{\mathbb{Y}}
\newcommand{\injtp}{\x \hat \otimes_{\varepsilon} \y}
\newcommand{\pin}{\|_{\varepsilon}}
\newcommand{\mc}{\mathcal}
\newcommand{\inter}{\left [0, 1\right ]}
\newcommand{\lir}{\lim_{r \to 1}}
\newcommand{\ha}{\mathfrak {H}}
\newcommand{\dom}[1]{D(#1)}
\newcommand{\mquad}[1]{\quad\text{#1}\quad}
\makeatletter
\newcommand{\normt}{\@ifstar\@normts\@normt}
\newcommand{\@normts}[1]{%
  \left|\mkern-1.5mu\left|\mkern-1.5mu\left|
   #1
  \right|\mkern-1.5mu\right|\mkern-1.5mu\right|
}
\newcommand{\@normt}[2][]{%
  \mathopen{#1|\mkern-1.5mu#1|\mkern-1.5mu#1|}
  #2
  \mathclose{#1|\mkern-1.5mu#1|\mkern-1.5mu#1|}
}
\makeatother

\thanks{Version of \today}

\begin{abstract}Motivated by models of signaling pathways in B lymphocytes, which have extremely large nuclei, we study the question of how reaction-diffusion equations in thin $2D$ domains may be approximated by diffusion equations in regions of smaller dimensions. In particular, we study how transmission conditions featuring in the approximating equations become integral parts of the limit master equation. We device a scheme which, by appropriate rescaling of coefficients and finding a common reference space for all Feller semigroups involved, allows deriving the form of the limit equation formally. The results obtained, expressed as convergence theorems for the Feller semigroups, may also be interpreted as a weak convergence  of underlying stochastic processes.

\end{abstract}

\maketitle

\newcommand{\oper}{\mathfrak R_r}
\newcommand{\opern}{\mathfrak R_{\rn}^\mho}
\newcommand{\brn}{\mbox{$\Delta^\mho_{\rn}$}}
\newcommand{\bro}{\mbox{$\Delta_{\rn}$}}
\newcommand{\rn}{r}
\newcommand{\cern}{C\hspace{-0.07cm}\left[\rn, 1\right ]}
\newcommand{\cernbez}{C\left[\rn, 1\right ]}
\newcommand{\cep}{C\hspace{-0.07cm}\left[ 0, 1\right ]}
\newcommand{\copi}{C[0,\pi]}
\newcommand{\CP}{\mbox{$C_p[0,2\pi]$}}
\newcommand{\cerec}{C\hspace{-0.07cm} \left ([0,\pi]\times [r,1]\right)}
\newcommand{\cerecbez}{C \left ([0,\pi]\times [r,1]\right)}
\newcommand{\cerecdwa}{C^2\hspace{-0.07cm} \left ([0,\pi]\times [r,1]\right)}
\newcommand{\cerecj}{C\hspace{-0.07cm} \left ([0,\pi]\times \left [0 ,1\right ]\right)}
\newcommand{\xprim}{C_\theta (UR)}
\newcommand{\ie}{i.e., }
\newcommand{\rla}{R_\lambda}
\newcommand{\grubex}{\mathbb X}
\newcommand{\Jcg}[1]{\left ( #1 \right )_{i=1,...,N}} 
\newcommand{\LB}{\mbox{$\Delta_{\text{{\tiny\textsf LB}}}$}}







\section{Introduction} 
\subsection{General remarks}
In modeling biological processes one often needs to take into account different time-scales of the processes involved \cite{banalachokniga,knigaz}. This is in particular the case when one of the components of the model is a diffusion which in certain circumstances may transpire to be much faster than other processes. For example, 
in the Alt and Lauffenberger's \cite{alt} model of leucocytes reacting to a bacterial invasion by moving up a gradient of some chemical  attractant produced by the bacteria (see Section 13.4.2 in \cite{keenersneyd}) a system of three PDEs is reduced to one equation provided bacterial diffusion is much smaller than the diffusion of leukocytes or of chemoattractants (which is typically the case). Similarly, in the early carcinogenesis model of Marciniak--Czochra and Kimmel \cite{osobliwiec,knigaz,nowaaniaprim,markim1,markim2},  a system of two ODEs coupled with a single diffusion equation (involving Neumann boundary conditions) is replaced by a so-called shadow system of integro-differential equations with ordinary differentiation, provided diffusion may be assumed fast. 

In this context it is worth recalling that one of the fundamental properties of diffusion in a bounded domain is that it `averages' solutions (of the heat equation with Neumann boundary condition) over the domain. As it transpires, it is this homogenization effect of diffusion, when coupled with other physical or biological forces that leads to intriguing singular perturbations  \cite{osobliwiec}; we will see this effect in the analysis presented below also.

The situation this paper is devoted to is not, from the outset, the one in which one component is faster than the other components, but may be translated, as we shall see later, to such a case via an isomorphism. Namely, in modeling dynamics of active and inactive kinases diffusing in a cell, viewed as a 3D ball, by means of a reaction-diffusion equation it has been noted that there are cells, e.g. B-lymphocytes, that have extremely large nuclei (see e.g. \cite{dlajima,thinaccepted}). As a result, kinases diffuse in a very thin layer between the nucleus and the cell membrane, and it seems reasonable to think that the state-space of the process is the unit sphere rather than a spherical shell. The question is whether, besides numerical simulations, there is a rigorous  limiting procedure that would justify this heuristic reasoning. 

The procedure should in particular reveal what happens with the boundary and/or  transmission conditions. To explain: in the model of kinase activity due to Ka\'zmierczak and Lipniacki \cite{kazlip} (see also \cite{bogdan2010}), the processes of phosporylation (activation) and  dephosphorylation (deactivation) of kinases are described by different means. Deactivation takes place when kinases meet phosphatases; since this process takes place in the interior of the cell, it is described by a reaction term in the master equation. On the other hand, kinases are activated only when they touch an active receptor located at the cell membrane: therefore, this process is modeled by a boundary condition. However, if our heuristic reasoning is valid and diffusion in thin $3$D layers should be approximated by a diffusion on a $2D$ surface (sphere), we face the question of what happens with such boundary conditions in the limit, because on the sphere no boundary conditions are needed or possible. From the biological point of view it is clear that such boundary conditions should somehow be transferred to the master equation, because in the $2D$ model kinases meet receptors in the interior of the state space and not on its boundary. These intuitions are also confirmed by numerical simulations (see \cite{thinaccepted}). From the mathematical viewpoint, on the other hand, the process of transferring boundary conditions into new terms of the master equation is still quite intriguing; see \cite{dlajima,thinaccepted} and our Figure \ref{kurcze}. 
\begin{center}
\begin{figure}
\includegraphics[scale=2.1]{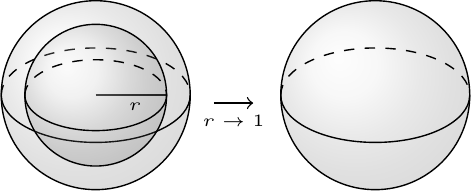}
 \caption{In \cite{dlajima,thinaccepted}, under assumption of axial symmetry, it has been shown that the Cauchy problem for the reaction-diffusion equation $u_t = \Delta u + f(u)$ where $f$ is a Michaelis-Menten function, in the region between two spheres, the inner of radius $r<1$, the outer with radius $1$, supplied with Neumann boundary condition on the inner sphere and the Robin-type boundary condition $u_{\nu} = R (1- u)$ on the outer sphere (where $R$ is a non-negative constant, and $u_{\nu}$ is the outward normal derivative) is well-posed in the space of continuous functions. Moreover, it is shown that as the radius of the inner sphere converges to $1$, solutions to the corresponding problems converge to solutions to the equation $u_t = \Delta_{LB} + R (1-u) + f(u)$ on the sphere, where  $\Delta_{LB}$ is the Laplace--Beltrami operator.} \label{kurcze}
\end{figure}
\end{center} 

\vspace{-0.3cm}

To be sure, the example of $B$ lymphocytes described above, is a member of a much larger class of phenomena.  Biology abounds in cases in which domain of interest is characterized by two different linear dimensions, one small (for instance, describing thickness), the other large. Examples include domains bounded by locally parallel surfaces, such as concentric spheres, coaxial tubes or simply parallel planes. All these are relevant for biological applications. For, within living cells there are numerous membranes, like  mitochondrial, nuclear, endoplasmic reticulum, endosomal,  Golgi, or cell membranes. All of them have thickness of the order of several nanometers, much smaller than the radius of the organelle or cell they encompass \cite{alberts}. Also in multicellular scale there are organs like veins, arteries, intestine, bladder, skin or endothelium, formed by a relatively thin layers of cells, or even monolayer as in the case of endothelium. These organs are all characterized by a big difference between the smallest and the largest linear dimensions.


\subsection{Details of the model considered; transmission conditions}\label{dotm} 

This paper is devoted to a case study in two dimensions with aesthetically pleasing, simple geometry and involving transmission conditions rather than boundary conditions: we imagine particles diffusing in two thin annular layers around the unit circle, and think of the circle as a semi-permeable but homogeneous membrane. In either of these layers chaotic movement of the particles is governed by the diffusion equation with $2D$ Laplace operator and possibly a reaction term. We assume that upon touching the lower border of the smaller annulus the particles are reflected, and that the same mechanism governs their behavior on the upper border of the larger annulus. Therefore,  we impose Neumann no-flux boundary conditions for the concentration $u$ of particles there: 
\[ \frac {\partial u}{\partial \nu } (p) = 0 , \quad \text{ for points } p \, \text{ on the boundary}\]
($\nu$ is the normal vector pointing outwards). We also assume that the membrane of the form of  the unit circle is homogeneous: its permeability does not vary in space. However, its permeability coefficient, say $\alpha>0$ for filtering from the upper annulus to the lower annulus is in general different from the coefficient, say $\beta$, for filtering from the lower annulus to the upper annulus. These assumptions are expressed quantitatively in the following system of transmission conditions: 
\begin{align}
\frac {\partial u}{\partial \nu } (p+) = \alpha (u(p+)- u(p-)),\nonumber \\ 
\frac {\partial u}{\partial \nu } (p-) = \beta  (u(p+)- u(p-)). \label{pierwszeb}
\end{align}
Here, we write $p-$ to denote a point on the unit circle that lies `right below' the membrane, and $p+$ to denote the point that lies `right above' the membrane. Hence, $u(p-)$ and $u(p+)$ are the concentrations measured at the same point but on the different sides of the  membrane. Intuitively, what these boundary conditions are saying is that a Brownian particle diffusing in the upper annulus is being many times reflected from the membrane, the so-called local L\'evy time it spends `at the membrane' is measured, and once an exponential time, with parameter $\alpha$, with respect to the L\'evy time elapses, the particle filters through the membrane to the lower annulus. Here, it behaves similarly: the time it spends `at the membrane' before filtering to the upper annulus is (independent and) exponential with parameter $\beta$. See \cite{knigaz} p. 66, \cite{zmarkusem} p. 669, \cite{bobmor,nagrafach,lejayn} and references given there for more information and more literature on transmission conditions of type \eqref{pierwszeb}; in particular, \cite{knigaz} discusses relation of such transmission conditions to the classical Feller--Wentzell boundary conditions.  
\begin{center}
\begin{figure}
\includegraphics{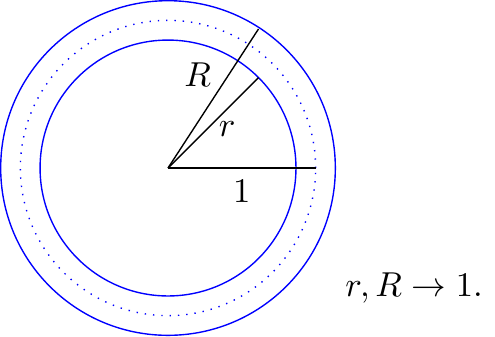}
 \caption{Two thin annuli separated by a circular semi-permeable membrane (depicted as the dotted circle).} \label{fig2}
\end{figure}
\end{center} 

\vspace{-0.3cm}
To repeat, it is intuitively clear that if the annuli are `very thin' this process should be modeled by $1D$ and not $2D$ equations. We will make this statement more precise by showing that, as `thickness' of the annuli converges to zero, solutions to the Cauchy problem related to a $2D$ diffusion equation with boundary and transmission conditions specified above converge, in the sense clarified below, to solutions of equations governing diffusion in regions of smaller dimensions.

To explain this result in more detail, we start by noting that it is intuitively pleasing that in the limit the Laplace operator governing diffusion in the thin layers should be replaced by the Laplace--Beltrami operator describing diffusion on the circle, and that, since in the approximating processes we had two separate regions of diffusion: above and below the membrane, in the limit we should have diffusion on the upper and lower parts of the membrane (see Figure \ref{pstanow1} further down). Hence, the main question is: what happens with the transmission conditions? What is their role in the limit process? 

Extending the principle discussed above and expounded in \cite{dlajima,thinaccepted}, saying that in the thin layer approximation boundary conditions become integral parts of the limit master equation, we will show that also  the
\begin{equation}\label{zasada} \boxed{\text{\textbf{transmission conditions `jump into' the limit master equation.
}}} \end{equation}  More specifically, in Thm. \ref{main1} we will show that (for a suitable choice of parameters, see Remark \ref{parametry})
the limit equation is of the form 
\begin{equation}
\frac {\ud } {\ud  t} \binom {u_+}{u_-} = \begin{pmatrix}\LB & 0 \\ 0 & \LB \end{pmatrix}  \binom {u_+}{u_-}  + \begin{pmatrix} -\alpha & \alpha \\ \beta & - \beta \end{pmatrix} \binom {u_+}{u_-} \label{prawieconvex} \end{equation}
where $u_+$ and $u_-$ are concentrations of particles on the upper and lower sides of the circular membrane, and $\LB$ is the Laplace--Beltrami operator (see e.g. our Section \ref{spacex}). In Theorems \ref{main2} and \ref{main3}, the same principle will be established in two related scenarios (see Section \ref{msg} for more details, and Section \ref{cr} for a discussion of all these three results). 
Formally, the theorems obtained here concern convergence of Feller semigroups of operators governing the approximating and the limiting equations (see below for more on Feller semigroups).

A similar result has been established in \cite{3Dlayers} where the case of diffusion in three dimensional layers separated by a flat membrane was studied in detail. In particular, it was shown there that the principle just mentioned, i.e. that permeability coefficients of the membrane become integral parts of the limit equation, is robust to a change in the way the particles filter through the membrane (i.e., the membrane might be partly sticky): while such a change influences the initial condition of the limit Cauchy problem, it does not alter the limit master equation itself.

\subsection{Mode of convergence, and the form of the limit equation}

An explanation of the mode of convergence is here in order: our point of departure is the statement that a sufficiently regular stochastic process may be described by the corresponding Feller semigroup of operators \cite{kniga,ethier,feller,kallenberg}, defined in the space of continuous functions on the state-space of the process. 

To recall, a family $(T(t))_{t\ge 0}$ of 
operators in a Banach space $\x$ is said to be 
a strongly continuous semigroup if the following three conditions are satisfied 
\begin{itemize} 
\item [(i)] $T(0)f = f, f \in \x$,
\item [(ii)] $T(t)T(s) = T(t+s), s,t\ge 0,$
\item [(iii)] $\lim_{t\to 0+} \|T(t)f - f\| =0, f\in \x$ (where $\|\cdot\|$ is the norm in $\x$).
\end{itemize} 
Each strongly continuous semigroup $(T(t))_{t\ge 0}$ is uniquely characterized by its infinitesimal generator $A$ defined by 
\[ Af = \lim_{t\to 0+} t^{-1} (T(t)f - f),\] 
on the domain $\dom{A}$ composed of $f\in \x$ for which this limit exists (in the sense of the $\x$ norm). The fundamental Hille--Yosida--Feller--Phillips--Miyadera Theorem characterizes operators $A$ that are generators of strongly continuous semigroups \cite{abhn,kniga,engel,ethier,feller,kallenberg}. In what follows, the semigroup generated by $A$ will be denoted $\sem{A}$. 

A conservative Feller semigroup $(T(t))_{t\ge 0}$ in the space $C(S)$ of continuous functions on a compact metric space  $S$ is a strongly continuous semigroup in $C(S)$ such that $T(t)f \ge 0$ whenever $f\ge 0$, and $T(t)1_S = 1_S, t \ge 0$, where $1_S (x) = 1, x \in S$. A characterization of conservative Feller generators may be found e.g. in \cite{kniga,ethier,kallenberg,liggett}, see also our Appendix (Section \ref{appendix}): $A$ is a Feller  generator iff it is densely defined, and satisfies the maximum principle and the so-called range condition; additionally $1_S $ must belong to $\dom{A}$ and we need to have $A1_S=0$.

With a conservative Feller semigroup in $C(S)$ there is (in a sense unique) honest stochastic process $X(t), t\ge 0$ with values in $S$ such that 
\[ T(t) f (x) = E_x f(X(t)) , \qquad t\ge 0, x \in S, f \in C(S),\]
where $E_x$ denotes the expected value conditional on $X(0)=x$. On the other hand, semigroups are related to Cauchy problems \cite{abhn,engel}: the Cauchy problem in a Banach space $\x$: 
\begin{equation}\label{cauchy}    u'(t) = A u (t) , \qquad t \ge 0, u(0)= x\in \x,\end{equation}
is well-posed iff $A$ is a semigroup generator. If this is the case, the function $t\mapsto u(t)= T(t)x$ (said to be a  trajectory of the semigroup) is a (mild) solution to the Cauchy problem. In the case of Feller semigroups, \eqref{cauchy} is the Kolmogorov backward equation for the related process.

Coming back to our main subject: the Banach spaces where the Feller semigroups governing diffusion equations involved in the thin approximation are defined, vary with the thickness of the domain. Hence, we consider their isomorphic images in a single reference space, corresponding to fixed thickness, and show that these images govern diffusion equations with faster and faster radial diffusion.  When seen in this reference space, therefore, the semigroups' trajectories gradually lose dependence on the radius. However, since they differ in the upper and lower annuli, in the limit they may be identified with two functions depending on the angle (and time). The reference space is the space of continuous functions on two annuli of fixed thickness, and the limit semigroup acts in the subspace of functions 
that, when restricted to either of these annuli, depend not on the radius. This subspace is  clearly isometrically isomorphic to the space of pairs of continuous functions on the circle. In other words, while the state-space of the approximating Feller processes is composed of two annuli, the state-space of the limit process is composed of two circles, or two sides of the same circle.

We may now comment on the form of the limit equation \eqref{prawieconvex}: this equation expresses an intriguing fact, which could perhaps have been expected on the basis of biological background, that  permeability coefficients featuring in \eqref{pierwszeb} in the limit become jump intensities of the process (compare \cite{zmarkusem}*{Remark 2.1}). For, a~typical realization of the random process governed by \eqref{prawieconvex} is as follows: a~particle diffusing on the lower side of the circle jumps, after an exponential time with parameter $\alpha$, to the point directly on the opposite side of the circle (i.e. the point with the same coordinates, but lying on the other side of the membrane) and continues diffusing on the upper side until, after another, independent, exponential time with parameter $\beta$, it jumps back to the lower side, and so on.   It is worth noting that such processes are closely related to piecewise deterministic Markov processes of M.H.A Davis \cite{davis,davisk,davisk2,rudnickityran}, random evolutions of R.J. Griego and R. Hersh \cite{ethier,gh1,gh2,pinskyrandom} and to randomly switching diffusions \cite{dwoje,ilin,yin}; for a semigroup theoretic context see  \cite{convex,knigaz}.

\subsection{Three scanarios}\label{msg}
\begin{center}
\begin{figure}
\includegraphics[scale=0.9]{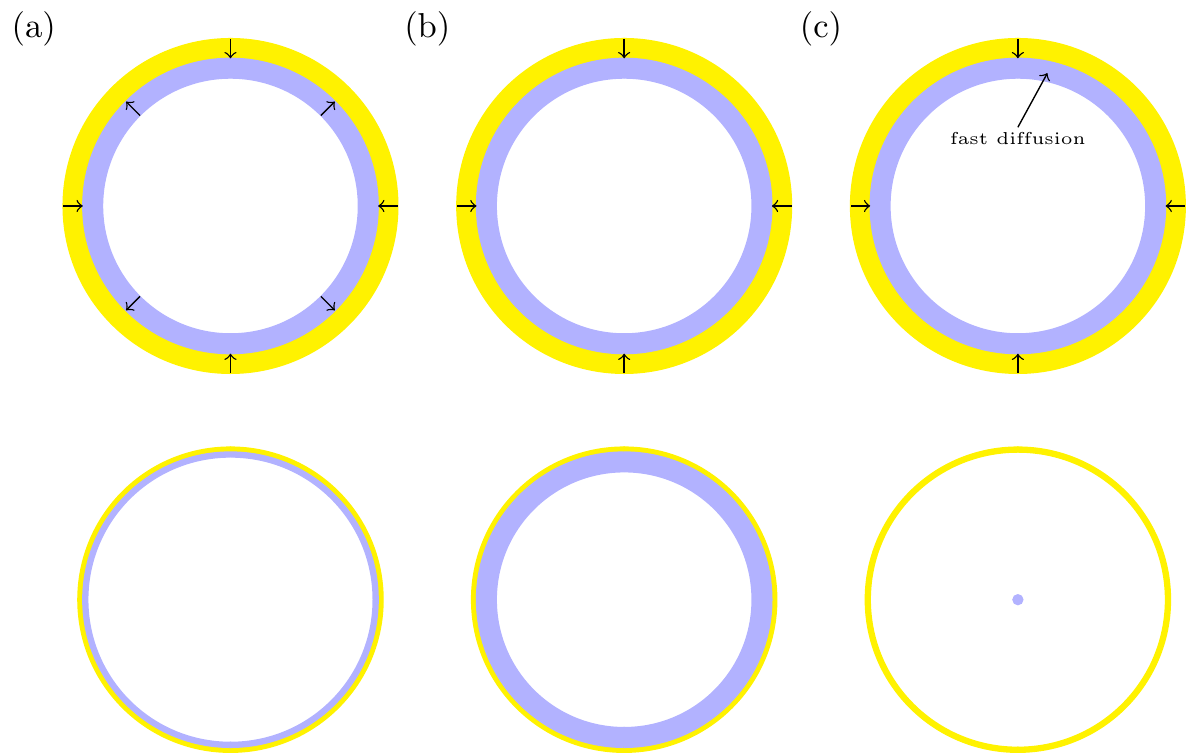}
 \caption{Different limit state-spaces resulting from different approximations: In the case (a) both layers are thin, and the limit state-space is composed of the upper and the lower sides of the unit circle. In the case (b) only the upper layer is thin, and thus the limit state-space is composed of the union of the unit circle and the lower layer/annulus. In the case (c), even though the lower layer is thick, diffusion there is fast and thus the limit state-space is the union of the unit circle and the point formed by compounding all elements of the lower annulus.} \label{fig4}
\end{figure}
\end{center} 

\vspace{-0.5cm}

As already mentioned, in the present paper we deal with different geometry than in \cite{3Dlayers}: we focus on two dimensions and choose the example of a circular membrane as a case study. A particular goal of this study is to establish the fact that in the thin layer approximation transmission conditions become integral parts of the limit equation in three natural scenarios similar to that described in Section \ref{dotm}.

To explain, in Section \ref{dotm} we assumed that both the upper and the lower layers are thin, and in the limit obtained two coupled PDEs on a surface of dimension one (see Figure \ref{fig4} (a)). However, one may also think of a `thick' two-dimensional region bordering a `thin' two-dimensional region (see Figure \ref{fig4} (b)). In this case, in the limit we face a PDE on a two-dimensional region coupled with a PDE on a one-dimensional surface (Thm. \ref{main2}). Interestingly, here also one of permeability coefficients becomes an integral part of the main equation and describes jumps from the one-dimensional surface to the two-dimensional region.  

Moreover, a diffusion in a thin layer may be accompanied by a very fast diffusion in the bordering `thick' two-dimensional region (see Figure \ref{fig4} (c)). Then the limit equation involves a PDE on a surface coupled with an ODE: the fast diffusion averages out the concentration in the two-dimensional region, and thus, at any particular time, the concentration may be described by a single real number. Again, permeability coefficients become jump intensities between points of the surface and the isolated point formed by lumping all the points of the `thick' two-dimensional region (see Thm. \ref{main3}).

\subsection{Concluding remarks (for the introduction)}
All the main results are expressed as convergence theorems for related Feller semigroups, and thus, as in the Trotter–Sova–Kurtz–Mackevi\u cius Theorem (see e.g. \cite{kallenberg} p. 385),  may be interpreted in terms of a weak convergence of the stochastic processes involved.  See also \cite{thinaccepted} for a discussion of relations between the  results of the type obtained here and (a) the Conway--Hoff--Smoller homogenization principle (see \cite{knigaz,chs,smoller}), (b) the Freidlin--Wentzel averaging principle (see \cite{fw,fwbook,freidlin}), (c) the averaging principle discussed in \cite{bobmor,nagrafach,zmarkusem}, and (d) the theory of thin layer approximation originating from the seminal paper of Hale and Raugel \cite{hale}.  In particular, transmission conditions, which are the main subject of the present paper, seem to play no role in either the 
Conway--Hoff--Smoller homogenization principle or the Freidlin--Wentzel averaging principle or in the 
existing literature on thin layer approximation. For a comprehensive review of the growing literature on the thin layer approximation see e.g. the upcoming \cite{brzezniakpraca}.

Finally, we mention that, even though we focus on (linear) semigroups, and thus on linear (diffusion) equations, semi-linear equations and in particular reaction-diffusion equations with Lipschitz continuous non-linearity may be treated as a direct application of the results obtained here: the main theorem of 
\cite{zmarkusem2b} says that convergence of the semigroups describing the linear parts of semi-linear equations implies convergence of the solutions to the entire semi-linear equations. The theory applies also to locally Lipschitz non-linearities provided that additional conditions of M\"uller type are satisfied (see e.g. \cite{knigaz,thinaccepted}). This is vital because it is these non-linearities that are responsible for describing important and interesting phenomena accompanying reaction-diffusion equations.  

The paper is organized as follows: in Section \ref{agt} we prove a master generation theorem for the semigroups featuring in the paper; in Sections \ref{sec:3}--\ref{sec:5} 
we present the main theorems devoted to the three scenarios of Section \ref{msg}; these theorems are then discussed in Section \ref{cr}. Section \ref{appendix} collects auxiliary one-dimensional generation results. 

\newcommand{\spacec}{C(\mc R_{r,R})}
\newcommand{\spacecr}{C(\mc R_{r,1+\gamma (1-r)})}
\section{A generation theorem}\label{agt}
Given $r\in(0,1)$ and $R>1$ consider the annulus in the $(x,y)$ plane where 
\[ r \le \sqrt {x^2+ y^2} \le R, \]
divided in two parts by a semi-permeable membrane of the form of the unit circle where $x^2+y^2 =1$ (see Figure \ref{fig2}).

In polar coordinates, say $(\rho, \phi)$, this region is the Cartesian product 
\[ \mc R_{r,R} = V_{r,R} \times [0,2\pi) \]
of \[V_{r,R} = [r,1-]\cup [1+,R]\] and $[0,2\pi)$. Alternatively (see Figure \ref{fig3}), $\mc R_{r,R}$ may be seen as a sum of two sub-rectangles:
\[ \mc R_{r,R} =  \mc R^-_r \cup \mc R^+_R \] 
where \[ \mc R^-_r = [r,1-] \times [0,2\pi) \mquad{ and } \mc R^+_R = [1+,R] \times [0,2\pi).\] The segments 
$\{1-\}\times [0,2\pi)\subset \mc R^-_r$ and $\{1+\}\times [0,2\pi)\subset \mc R^+_R$ represent the points `right below' and `right above' the membrane (we use notation similar to that introduced in Introduction for $p+$ and $p-$). Since $\mc R^-_r$ and $\mc R^+_R$ are disjoint, a function $v:\mc R_{r,R} \to \R$ is continuous iff it is continuous on either of these sub-rectangles (separately).  
\begin{center}
\begin{figure}[h]
\includegraphics{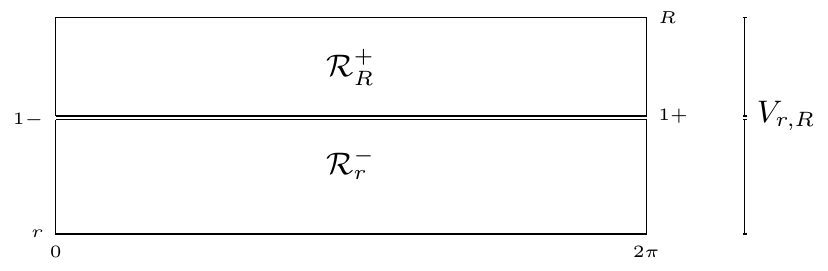}
 \caption{Two thin annuli separated by a semi-permeable membrane in polar coordinates.} \label{fig3}
\end{figure}
\end{center} 

\vspace{-0.5cm} To recall, in polar coordinates the Laplace operator is given by 
\begin{equation}\label{laplace} \Delta v(\rho, \phi) = \frac{\partial^2 v}{\partial \rho^2 }(\rho, \phi)  + \frac 1 \rho \frac{\partial v}{\partial \rho } (\rho, \phi) +  \frac 1 {\rho^2} \frac{\partial^2 v}{\partial \phi^2 } (\rho, \phi). \end{equation}  
To find a realization of this operator that will suit our purposes, we will work in the space $\spacec$  of continuous functions $v$ on $\mc R_{r,R}$ such that 
\[ v(\rho, 0) = v(\rho, 2\pi)\coloneqq \begin{cases} \lim_{\mc R^-_r\ni (\rho',\phi) \to (\rho, 2\pi)} v(\rho',\phi),& \rho \in [r,1-],\\
\lim_{\mc R^+_R \ni (\rho',\phi) \to (\rho, 2\pi)} v(\rho',\phi),& \rho \in [1+,R];\end{cases} \]
(existence of the limit featuring here is part of the characterization of $v\in \spacec$). The latter space, when equipped with the supremum norm, may be identified with the injective tensor product $C(V_{r,R}) \otimes_\eps \CP$ (see e.g. \cite{ryan}) where $C(V_{r,R})$ is the space of continuous functions on $V_{r,R}$, and $\CP$ is the space of continuous functions $g $ on $[0,2\pi]$ such that $g(0)=g(2\pi)$. This is to say that (a) the supremum norm of $\spacec$ coincides with the tensor injective norm inherited from $C(V_{r,R})$ and $\CP$ and (b) each member of $\spacec$ may be approximated with arbitrary accuracy by linear combinations of simple tensors, i.e., of vectors of the form $v=f \otimes g$:
\[ v(\rho, \phi) = f(\rho) g(\phi), \qquad f \in C(V_{r,R}), g \in \CP, \rho \in V_{r,R}, \phi \in [0,2\pi]. \]

We are now ready to present the main generation theorem of this section. 

\begin{thm} \label{mastergen} Let $\alpha,\beta\ge 0$ and $\kappa>0$ be given constants. Furthermore, let $\mc A\coloneqq \mc A_{\alpha,\beta,\kappa, r,R}$ be the operator in $\spacec$ given by 
\[ \mc Av(\rho, \phi) = \begin{cases} 
  \frac{\partial^2 v}{\partial \rho^2 }(\rho, \phi)  + \frac 1 \rho \frac{\partial v}{\partial \rho } (\rho, \phi) +  \frac 1 {\rho^2} \frac{\partial^2 v}{\partial \phi^2 } (\rho, \phi), & (\rho, \phi )\in \mc R^+_R, \\
  \kappa \left [ \frac{\partial^2 v}{\partial \rho^2 }(\rho, \phi)  + \frac 1 \rho \frac{\partial v}{\partial \rho } (\rho, \phi) +  \frac 1 {\rho^2} \frac{\partial^2 v}{\partial \phi^2 } (\rho, \phi)\right ], & (\rho, \phi )\in \mc R^-_r ,
\end{cases}
\]
for all $v$ of class $C^2$ (in either of $\mc R^+_R$ and $\mc R^-_r$, separately), such that 
\begin{equation}
\mc A v(\rho,0)=\mc A v (\rho,2\pi), \qquad \rho \in V_{r,R},\label{brzegi}
\end{equation}
and
\begin{align}
\frac{\partial v}{\partial \rho } (1+,\phi) & = \alpha [v(1+,\phi) - v(1-,\phi)], 
\label{war-unki}\\
\frac{\partial v}{\partial \rho } (1-,\phi) & = \beta [v(1+,\phi) - v(1-,\phi)], \nonumber 
\\ \frac{\partial v}{\partial \rho } (r,\phi)&= \frac{\partial v}{\partial \rho } (R,\phi)=0, \nonumber \qquad \phi \in [0,2\pi]. \end{align}
Then, $\mc A$ is closable, and its closure generates a conservative Feller semigroup in $\spacec$.  
\end{thm}

For the proof of this theorem, we need information on tensor product semigroups and on the Dorroh multiplicative perturbation theorem. Concerning the second of these (see e.g.  \cite{dorroh,engel}), in the simple form needed in the case of our interest, Dorroh's theorem says that if $S$ is a compact space, an operator $A$ is the generator of a Feller semigroup in $C(S)$, and $h\in C(S)$ is a positive function, then $f \mapsto h\, Af$ with domain $\dom{A}$ is a generator also. The stochastic process related to this new generator is a time-changed process of the operator $A$ (see \cite{convex}, \cite[Chapter 49]{knigaz} and \cite[p. 278]{rogers} for more details and more bibliography). 

Turning to the information on tensor product semigroups, suppose that $\sem{A}$ and $\sem{B}$ are two strongly continuous semigroups of contractions, generated by $A$ and $B$, in $C(V_{r,R})$ and $\CP$, respectively. Then, for $f \in C(V_{r,R})$ and $g \in \CP $, we may define 
\[ S(t) (f \otimes g) = \left (\e^{tA } f\right ) \otimes \left (\e^{tB } g \right ), \qquad t \ge 0. \] 
Since the set of linear combinations of simple tensors is dense in the space $\spacec$,
and the norm in $\spacec$ is the injective tensor norm inherited from $C(V_{r,R})$ and $\CP$, operators $S(t)$ may be extended to contractions defined on the the entire $\spacec$ (see \cite{ryan}). Moreover (see \cite[pp. 21-24]{nagel}), they form a strongly continuous semigroup; the latter semigroup is often termed the injective tensor product semigroup and  denoted \semt{A}{B}.  It is easy to see that if $\sem{A}$ and $\sem{B}$ are conservative Feller semigroups, then so is $\semt{A}{B}.$

A direct calculation shows that for $f\in D(A) $ and $g \in D(B)$, the function $v= g \otimes f$ belongs to the domain $D(\mathfrak A)$ of the generator $\mathfrak A$ of this semigroup, and 
\begin{equation}\label{tensorg} \mathfrak A v = (A f) \otimes  g + f \otimes (B g). \end{equation} 
It can be proven also that the set of linear combinations of such tensors is a core for $\mathfrak A$ (see for example Proposition at page 23 in  \cite{nagel}). 
The latter statement means that the closure of $\mathfrak A$ restricted to such linear combinations is the entire $\mathfrak A.$ Because of \eqref{tensorg}, the generator $\mathfrak A$ is often written as 
\[ \mathfrak A = A\otimes I_{C(V_{r,R})} + I_{\CP} \otimes B, \]
where $I_X$ is the identity operator in a Banach space $X$. 

\begin{proof}[Proof of Theorem \ref{mastergen}]
Let $B=\frac {\ud^2}{\ud \phi^2}$ be the operator in $\CP$ with domain composed of twice continuously differentiable functions $g$ such that $g'' \in \CP$. Then $B$ is the generator of a conservative Feller semigroup in this space. Moreover, let $A^I= A^I_{\alpha,\beta,\kappa,r,R}$ be the generator of the Feller semigroup described in Lemma \ref{lem2} and Proposition \ref{prop1} (in Appendix), and let $\mathfrak A$ be the generator of the tensor product semigroup $\semt{A^I}{B}$. 

For $v=f\otimes g$, where $f\in \dom{A^I}$ and $g\in \dom{B}$, 
\begin{equation} \mathfrak Av = (\chi^2 f'' + \chi f')\otimes g + f \otimes g''= \chi^2 \left (\frac{\partial^2 v}{\partial \rho^2} + \frac 1\chi \frac{\partial v}{\partial \rho} + \frac 1{\chi^2} \frac{\partial ^2 v}{\partial \phi^2} \right ) \text{ on } \mc R_R^+ \label{gwiazdka}
\end{equation}
and on $\mc R_r^-$ the expression must be multiplied by $\kappa$; in the right-most term of this formula,  with a slight abuse of notation, $\chi \in C(V_{r,R})$ defined in Lemma \ref{lem2} is identified with $\chi \otimes 1_{[0,2\pi]}\in \spacec$, so that $\chi (\rho,\phi) = \rho$. Moreover, since $\frac{\partial v}{\partial \rho} = f'\otimes g$ and $f$ satisfies \eqref{bandtdlar}, conditions \eqref{war-unki} are met by $v$. The assumption $g \in \dom{B}$ implies similarly that $\mathfrak Av$ calculated above belongs to $\spacec$. Thus,  \eqref{gwiazdka} together with its counterpart in $\mc R_r^-$ reveals that $v$ belongs to $\dom{\mc A}$ and 
\( \mathfrak A v = \chi^2 \mc A v .\)
This formula extends to the set $\mc D$ of all linear combinations of simple tensors (with factors in $\dom{A}$ and $\dom{B}$). Hence, denoting $\mc A^0$ the operator $\mc A$ restricted to $\mc D$, we obtain
\begin{equation}\label{zwiazek} \mathfrak A v = \chi^2 \mc A^0 v , \qquad v \in \mc D=\dom{\mc A^0}.\end{equation}

Next, we note that $\chi$ is separated from $0$. It follows that for $v_n \in \mc D$, sequences $\jcg{v_n}$  and $\jcg{\mathfrak A v_n}$ converge iff so do $\jcg{v_n}$  and $\jcg{\mc A^0 v_n}$, and then $\gra \mathfrak Av_n = \chi^2 \gra \mc A^0 v_n$. As a result, since $\mathfrak A$ is closed, $\mathcal A^0$ is closable, and since $\mc D$ is a core for $\mathfrak A$, the domain of the closure $\overline{\mc A^0}$ coincides with $\dom{\mathfrak A}$, and we have 
\[ \mathfrak A v = \chi^2 \overline{\mc A^0} v , \qquad v \in \dom{\overline{\mc A^0}}=\dom{\mathfrak A}.\]
By Dorroh's theorem, since $\mathfrak A$ is the generator of a conservative Feller semigroup, so is $\overline{\mc A^0}.$ 

We are left with clarifying the relation between $\mc A$ and $\mc A^0$: we want to show that $\mc A$ is closable and $\overline {\mc A}= \overline {\mc A^0}.$ To this end, we note first that, arguing as in our Appendix, one may show that $\mc A$ satisfies the maximum principle. Therefore (see \cite[Lemma 2.1 p. 165]{ethier}) 
\begin{equation}\label{wn-mp} 
\|\lam v - \mc A v \| \ge \lam \|v\|, \qquad v\in \dom{\mc A}.\end{equation}
Since $\mc A$ is densely defined, it follows that (see \cite[Lemma 2.11 p. 16]{ethier}) $\mc A$ is closable and then \eqref{wn-mp} can be extended to all $v\in  \dom{\overline {\mc A}}$ (with $\mc A v$ replaced by $\overline{\mc A}v$. 

Clearly $\overline {\mc A^0} \subset \overline{\mc A}$ since the operator $\mc A$ extends $\mc A^0.$  
To show the opposite inclusion, take $v \in \dom{\overline{\mc A}}$ and consider $\lam v - \overline{\mc A}v\in C(\mc R)$ for some $\lam >0.$ Since $\overline{\mc A^0}$ is a Feller generator, there is a $v' \in \dom{\overline{\mc A^0}}$ such that $\lam v - \overline{\mc A^0}v' = \lam v - \overline{\mc A}v $ and this means that $\lam (v-v') - \overline {\mc A}(v-v') =0$. By \eqref{wn-mp} (extended to $v\in \dom{\overline{\mc A}}$), however, this means that $v=v'$, and in particular $v$ belongs to $\overline {\mc A^0}.$   \end{proof}

\section{Two thin layers: convergence as $r,R\to 1$}\label{sec:3}

\begin{center}
\begin{figure}
\includegraphics[scale=1.2]{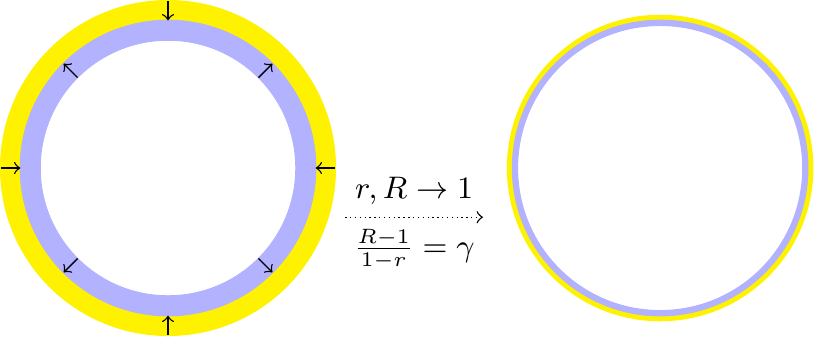}
 \caption{The approximating semigroups in Section \ref{sec:3} are defined in the spaces of continuous functions on the union of two (smaller and smaller) annuli; the limit semigroup  acts in the space of continuous functions on the union of the two sides of the unit circle.} \label{pstanow1}
\end{figure}
\end{center} 

\vspace{-0.3cm}

In this section, our goal is a thin layer convergence theorem for the semigroups obtained in Theorem \ref{mastergen}, when both radiuses converge to $1$: $r\to 1-$ and $R\to 1+$ (see  Figure \ref{pstanow1}). As discussed in \cite{dlajima} and \cite{thinaccepted}, for the limit to be non-trivial, transmission conditions \eqref{war-unki} need to be appropriately rescaled. Thus, we consider 
\[ \mc A_{(R-1)\alpha, (1-r)\beta,\kappa, r,R}. \]
Moreover, we assume that, while shrinking to $0$, the radial sizes of the `outer' and `inner' annuli are comparable: there is a constant $\gamma >0$ such that 
\begin{equation} \frac {R-1}{1-r} = \gamma .\label{rer}\end{equation}
(The same result may be obtained under the weaker condition  
\(\lim_{R\to 1+} (R-1) = \gamma \lim_{r\to 1-} (1-r) \), but, to simplify notations and analysis, we restrict ourselves to assumption \eqref{rer}.) This leads to considering the semigroups generated by 
the closures of \[\mc B_r 
\coloneqq \mc A_{(1-r)\alpha\gamma, (1-r)\beta,\kappa, r,1+\gamma(1-r)}.\]

\subsection{Reduction to analysis in a reference space}\label{reduction}

We also need to take into account the fact that each operator 
$\mc B_{r}$, 
and the related semigroup, is defined in a different space (namely, in the space $\spacecr$). Hence, there is a need to find a single, common reference space, and to define isomorphic images of the operators and semigroups $\sem{\overline{\mc B_r}}$ in this  space. There are of course many possible choices of such a space; we consider
\[ C(\mc R) = C(V) \otimes_\eps \CP, \]
where \[ V= [0,1-] \cup [1+, 2].\] Each member $v$ of $C(\mc R)= C(\mc R^- \cup \mc R^+)$ is a continuous function  on either \[ \mc R^-\coloneqq [0,1-] \times [0,2\pi] \mquad{ and } \mc R^+ \coloneqq [1+,2]\times [0,2\pi],\] separately, such that $v(\varrho,2\pi)=v(\varrho, 0)$ for all $\varrho \in V$. For any $r\in (0,1)$, this space is isometrically isomorphic to $\spacecr$, and a sample isomorphism is \[J_{r}: C(\mc R) \to \spacecr\] given by  \[ J_r u(\rho, \phi) = u(\varrho_r(\rho), \phi ), \qquad  u \in C(\mc R), (\rho, \phi) \in \mc R_{r,1+\gamma (1-r)},\]
where $\varrho_r (\cdot) $ maps $V_{r,1+\gamma (1-r)}$ onto $V$ as follows
\[ \varrho _r(\rho) =\begin{cases}  
\frac {\rho - 1}{\gamma(1-r)}+1, &\rho \in [1+,1+\gamma(1-r)],\\
\frac {\rho - r}{1-r},&\rho \in [r,1-] 
. \end{cases} \]
The inverse is, analogously, given by \[J_{r}^{-1} v(\varrho, \phi) = v(\rho_r(\varrho), \phi)
\qquad  v \in \spacecr, (\varrho, \phi) \in \mc R,\]
where $\rho_r (\cdot) $ maps $V$ onto $V_{r,1+\gamma (1-r)}$ via
\[ \rho_r(\varrho) =\begin{cases} 1+ \gamma(1-r)(\varrho - 1), &\varrho \in [1+,2], \\  \varrho (1-r)+r, &\varrho \in [0,1-]. \end{cases} \]
Thus, we are lead to studying 
\[ \lim_{r\to 1-} J_r^{-1} \e^{t\overline{\mc B_r}} J_r .\]

\newcommand{\brj}{\mc C_r} 
 \newcommand{\bRj}{\mc C_R}

\begin{thm} \label{gen1} Fix $r\in (0,1)$ and let $\brj$ be  the isomorphic image of $\mc B_r $ in $C(\mc R)$ via $J_r$.  Then, $\dom{\brj}$ is the set of functions $u \in C(\mc R)$ that are twice continuously differentiable in either of the rectangles $\mc R^+$ and $\mc R^-$, separately, and satisfy 
\begin{align}
\frac{\partial u}{\partial \rho } (1+,\phi) & = (1-r)^2\alpha \gamma [u(1+,\phi) - u(1-,\phi)],\label{war-un-ki}\\
\frac{\partial u}{\partial \rho } (1-,\phi) & = (1-r)^2 \beta [u(1+,\phi) - u(1-,\phi)],  \nonumber 
\\ \frac{\partial u}{\partial \rho } (0,\phi)&= \frac{\partial u}{\partial \rho } (2,\phi)=0, \nonumber \qquad \phi \in [0,2\pi]. \end{align}
Moreover, for $u\in \dom{\brj}$,  
\[ \brj u(\varrho, \phi)=
\begin{cases} 
\frac {1}{\gamma^2(1-r)^2}  \frac{\partial^2 u}{\partial \varrho^2 }(\varrho, \phi)  +  \frac {1}{\gamma \rho_r(\varrho) (1-r)}   \frac{\partial u}{\partial \varrho } (\varrho, \phi) +  \frac 1 {(\rho_r (\varrho))^2} \frac{\partial^2 u}{\partial \phi^2 } (\varrho, \phi),  & \\
 \frac {\kappa}{(1-r)^2}  \frac{\partial^2 u}{\partial \varrho^2 }(\varrho, \phi)  +  \frac {\kappa}{\rho_r(\varrho) (1-r)}   \frac{\partial u}{\partial \varrho } (\varrho, \phi) +  \frac \kappa {(\rho _r(\varrho))^2} \frac{\partial^2 u}{\partial \phi^2 } (\varrho, \phi), & \\
\end{cases}
\]
depending on whether $(\varrho, \phi )\in \mc R^+ $ (the upper line) or $(\varrho, \phi )\in \mc R^-$ (the lower line). 
\end{thm}

\begin{proof} By definition, if $u$ belongs to $\dom{\brj}$ then $v\coloneqq J_r u $ belongs to $\dom{\mc B_r}$, and in particular $v$ is twice continuously differentiable. Since $\rho_r(\cdot)$, as a linear function, is twice continuously differentiable as well, it is thus clear that also $u=J_r^{-1}v$ is twice continuously differentiable in either of $\mc R^+$ and $\mc R^-$. Moreover, 
\begin{align*}
\frac{\partial u}{\partial \varrho} (1+,\phi) &= (1-r)\gamma \frac{\partial v}{\partial \rho} (1+,\phi) = (1-r)^2  \alpha \gamma [v(1+,\phi) - v(1-,\phi)]\\&= (1-r)^2 \alpha \gamma [u(1+,\phi) - u(1-,\phi)], \qquad \phi \in [0,2\pi],
\end{align*} 
proving the first condition in \eqref{war-un-ki}; the remaining conditions are established similarly. Conversely, by symmetrical reasoning, if a $u$ is twice continuously differentiable and satisfies conditions \eqref{war-un-ki}, then its image $J_r u$ belongs to $\dom{\mc B_r}$. It follows that $\dom{\brj}$ is precisely the set of such functions $u$.  

Moreover, for such $u$ and for $(\rho, \phi) \in \mc R_r^-$,  $\mc B_r J_r u (\rho, \phi) $ equals
\[ \frac \kappa {(1-r)^2} \frac{\partial^2 u}{\partial \varrho^2} (\varrho_r (\rho), \phi) + 
\frac \kappa {\rho (1-r)} \frac{\partial u}{\partial \varrho} (\varrho_r (\rho), \phi)  + \frac \kappa {\rho^2} \frac{\partial^2 u}{\partial \varphi^2} (\varrho_r(\rho), \phi).\]
For  $(\rho, \phi) \in \mc R_{1+\gamma(1-r)}^+$ the formula is similar: the consecutive coefficients must be replaced by \[ \frac 1 {\gamma^2(1-r)^2}, \quad \frac 1 {\rho \gamma (1-r)} \mquad { and }  \frac 1 {\rho^2},\]
respectively. Since $\brj u = J_r^{-1}\mc B_r J_r u, $ this proves the rest of the theorem. 
 \end{proof}

Since the isomorphisms $J_r$ and $J_r^{-1}$ preserve the lattice structures of the spaces involved, and map $1_{\mc R}$ to $1_{\mc R_{r,1+\gamma (1-r)}}$ and \emph{vice versa}, combining Theorems \ref{mastergen}  and \ref{gen1} we see that for any $r\in (0,1)$, $\brj$ is closable and its closure is a Feller generator in $C(\mc R)$; the semigroup generated by $\overline{\brj}$ is $\left (J_r^{-1} \e^{\overline {\mc B_r}t} J_r, t\ge 0\right )_{t\ge 0}.$ Therefore, our task is that of studying
 \[ \lim_{r\to 1-}  \e^{t\overline{\brj}}\]
in $C(\mc R)$. 

\subsection{The underlying semigroup $\sem{\mc Q}$ and its properties}\label{wlasnosciq}
The following operator in $C(V)$ has an immediate bearing on the question of existence and form of this limit. Let $Qf(\varrho)=\gamma^{-2}f''(\varrho)$ for $\varrho \in [1+,2] $ and $Qf(\varrho)=\kappa f''(\varrho)$ for $\varrho \in [0,1-]$ on the domain composed of functions $f$ that are twice continuously differentiable in either of the subintervals forming $V$ and such that  $f'(0)=f'(1-)=f'(1+)=f'(2)=0$. A standard analysis shows that $Q$ is a conservative Feller generator. The related semigroup governs a Brownian motion on $V$ with different diffusion coefficients in each of the subintervals and  with reflecting barriers at $0,1-,1+$ and $2$; in particular, a traveler starting in $[0,1-]$ never enters $[1+,2]$ and \emph{vice versa}. Moreover, 
\begin{equation}\label{asyQ} \grat \e^{tQ}f = Pf, \qquad f \in C(V) \end{equation} 
where $Pf$ equals $\int_0^1 f$ or $\int_1^2 f$ on $[0,1-]$ and $[1+,2]$, respectively. 

The semigroup 
\[\left (\e^{t{Q}} \otimes_\varepsilon I \right )_{t \ge 0}= \semt{Q}{0}\]
where $I=I_{\CP}$ is the identity operator in $\CP$ and $0$ denotes the zero operator in this space, is a version of $\sem{Q}$ in $C(\mc R)$. Let \[\mc Q = Q\otimes I\] be the generator of this semigroup. Clearly, 
\begin{equation}\label{asymcQ} \grat \e^{t\mc Q}u = \mc P u, \qquad u \in C(\mc R) \end{equation} 
where \[ \mc P u (\varrho, \phi ) = \begin{cases} 
  \int_1^2 u(\varrho', \phi) \ud \varrho', &(\varrho, \phi) \in \mc R^+,\\
\int_0^1 u(\varrho', \phi) \ud \varrho', &(\varrho, \phi) \in \mc R^-. \end{cases}\]
We also recall that the set of $u$ of the form $u\coloneqq f \otimes g$ where 
 $f\in \dom{Q}$ and $g \in \CP$ is contained in $\dom{\mc Q}$ and is a core for $\mc Q$; moreover, if $u$ is of this form then  for $u_r\coloneqq f_r\otimes g$, where 
\[ f_r = f - (r-1)^2 (f(1+) - f(1-))\psi\] and 
\[ \psi (\varrho) = \begin{cases} 
\frac {\alpha \gamma}\pi (2-\varrho) \sin \pi \varrho , & \varrho \in [1+,2],\\
\frac {\beta}\pi \varrho \sin \pi \varrho ,& \varrho \in [0,1-],
\end{cases} \]
we have (see Remark \ref{onchi}, further down)
\begin{equation} \label{zwzq} u_r \in \dom{\brj}, \lim_{r\to 1-} u_r = u \mquad { and }  \lim_{r\to 1-} (1-r)^2\brj u_r = \mc Q u.\end{equation}

It is this connection between $\brj $'s and $\mc Q$, along with the property \eqref{asymcQ}, that allows inferring, under additional assumptions explained below, that the semigroups $\sem{\overline{\brj}}$ converge to a semigroup acting in the subspace of $C(\mc R)$, equal to the range of the projection $\mc P$:
\begin{equation}\label{spacesx} \x \coloneqq \text{Range}\, \mc P. \end{equation}  
More specifically, the singular perturbation theorem of T. G. Kurtz (\cite{ethier,kurtzper,kurtzapp} or \cite{knigaz}, Theorem 42.1) says that if, in this scenario, there is an operator $\mc O:\dom{\mc O}\to C(\mc R)$, with domain contained in $\x$ such that 
\begin{itemize} 
\item [(a)] for $u \in \dom{\mc O}$ there are $u_r \in \dom{\brj} $ such that $\grar u_r = u$ and $\grar \brj u_r = \mc O u,$
\item [(b)] $\mc {PO}$ is the generator of a semigroup in $\x$,
\end{itemize}
then 
\begin{equation}\label{kurtz} \grar \e^{t\overline{\brj} } u = \e^{t\mc{PO}} \mc P u , \qquad t>0, u \in C(\mc R),\end{equation}
and the limit is uniform in $t$ in compact subsets of $(0,\infty)$. For $u \in \x$ the relation is true also for $t=0$ and the limit is uniform in $t$ in compact subsets of $[0,\infty)$. 

\begin{rem} \label{onchi} The function $\psi$ defined above is twice continuously differentiable and satisfies 
\begin{align*} \psi '(0)&=\psi(1-)=\psi(1+)=\psi'(2) =0, \\ 
\psi'(1+)&=-\alpha \gamma \mquad {and} \psi'(1-) = -\beta.\end{align*}
Relations \eqref{zwzq} and the proof of Theorem \ref{main1} remain valid if this particular $\psi$ is replaced by any other function with these properties. We note that these properties imply:
\begin{equation}\label{propchi}\int_0^1 \psi'' = -\beta \mquad { and } \int_1^2 \psi'' =\alpha \gamma.  \end{equation}
  \end{rem}

\subsection{The space $\x$ and the limit semigroup}\label{spacex}
Examining the definition of $\mc P$, we conclude immediately that in our case $\x$ is the space of functions that, when restricted to $\mc R^+$ or $\mc R^-$, do not depend on $\varrho$. Thus, each member of $\x$ may be identified with a pair $\binom{g^+}{g^-}$ of members of $\CP$, \ie with two functions on the unit circle. 

For the operator $\mc O$ we take 
\[ \mc O \binom{g^+}{g^-} = \binom{\LB \, g^+}{\kappa \LB \, g^-} - \psi'' \otimes \binom{\gamma^{-2} (g^+-g^-)}{\kappa (g^+-g^-)}\]
with domain $\dom{\mc O} = \dom{\LB} \times \dom{\LB}.$ Here, $\dom{\LB}$  is the set of twice continuously differentiable functions $g\in \CP$ such that $g''\in \CP, $ and $\LB \, g = g''.$  The Laplace-Beltrami operator $\LB$ is the generator of the Feller semigroup governing Brownian motion on the unit circle. We note that, because of \eqref{propchi},
\[ \mc P \left (  \psi'' \otimes \binom{\gamma^{-2} (g^+-g^-)}{\kappa (g^+-g^-)}  \right )=\binom {\alpha \gamma^{-1} (g^+ - g^-)}{\kappa \beta(g^- - g^+)}.  
\]
Therefore, 
\begin{equation}\label{PO} \mc {PO} \binom{g^+}{g^-} = \binom{\LB \, g^+}{\kappa \LB\, g^-} +\begin{pmatrix}- \alpha \gamma^{-1} & \alpha\gamma^{-1}\\  \kappa \beta  & -\kappa \beta\end{pmatrix} \binom{g^+}{g^-}. \end{equation}

With the identification of $\x$ with $\CP\times \CP$ in mind, we have the following first main result of our paper. 

\begin{thm}\label{main1} We have 
\begin{equation*}\grar \e^{t\overline{\brj} } u = \e^{t\mc {PO}} \mc P u , \qquad t>0, u \in C(\mc R),\end{equation*} 
and the limit is uniform in $t$ in compact subsets of $(0,\infty)$. For $u \in \x$ the relation is true also for $t=0$ and the limit is uniform in $t$ in compact subsets of $[0,\infty)$.
\end{thm} 

\begin{proof} For $ u =\binom{g^+}{g^-}\in \dom{\mc Q}$, let 
\[ u_r =  \binom{g^+}{g^-} - (r-1)^2 \psi  \otimes \binom{g^+-g^-}{g^+-g^-}.\]
Then, because of the properties of $\psi$ listed in Remark \ref{onchi}, $u_r $ belongs to $\dom{\brj}.$ Moreover, $\grar u_r = u$ and $\grar \brj u_r = \mc O u$ (to see this note in particular that $\grar \rho_r(\varrho) =1$ uniformly in $\varrho \in V$).  Therefore, by Kurtz's Theorem, 
we are left with showing that $\mc{PO}$ is a generator. 

Since $\LB$ is a generator in $\CP$, the operator $\mc O'$ mapping $\binom{g^+}{g^-}$ to $\begin{pmatrix} {\LB \, g^+}\\{\kappa \LB \, g^-}\end{pmatrix}$ (with domain  $\dom{\LB} \times \dom{\LB}$) is a generator in $\CP\times \CP$. Hence, by the Phillips Perturbation Theorem (see e.g. \cite{kniga,engel}), $\mc {PO}$ is a generator also, as a bounded perturbation of $\mc O'$. This completes the proof. \end{proof}

\begin{rem}\label{parametry}For $\kappa=\gamma=1$, the operator $\mc{PO}$ agrees with the form of the limit equation announced in Introduction, see Equation \eqref{prawieconvex} there. In the approximating semigroups, parameter $\kappa$ accounts for differences in the `speed' of diffusion in the upper and lower parts of the layers, and its role in $\mc{PO}$ is similar. This parameter will be of greater importance in Section  \ref{sec:5}. 

The role of $\gamma$ is more interesting. As \eqref{rer} shows, for $\gamma >1$, the upper layer is thicker than the lower layer. Because of this, diffusing particles in the upper part touch the separating barrier less frequently than they would if the layers were of the same thickness, and therefore permeate through the barrier less often. This has a bearing on the limit equation: the jump intensity $\alpha \gamma^{-1} $ is smaller then $\alpha$ and the average time spent at the upper side of the membrane before jumping to the lower side of the membrane is longer. Analogously, for $\gamma <1$, the jump intensity is larger than $\alpha$, because in the approximating processes particles diffusing in the upper part hit the membrane more often than they would if $\gamma $ were $1$.  
\end{rem}

\section{A thin layer bordering a thick layer: convergence as $R\to 1$}\label{sec:4}
In this section, we consider the situation where the lower layer is thick and the upper layer is thin, and thus are interested in the limit of the semigroups constructed in Theorem \ref{mastergen}, as the outer radius $R\to 1+$ and the inner radius $r$ is kept constant (see  Figure \ref{pstanow2}). More specifically, we consider the limit as $R\to 1+$ of the semigroups generated by the closures of 
\[ \mc B_R\coloneqq \mc A_{(R-1)\alpha, \beta,\kappa, r,R}. \]
\begin{center}
\begin{figure}
\includegraphics[scale=1.2]{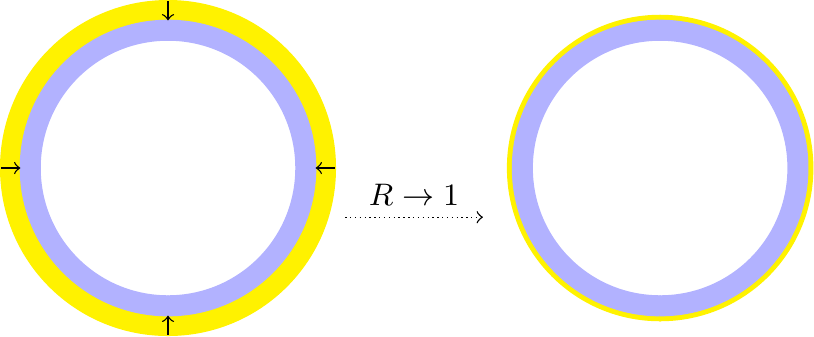}
 \caption{In Section \ref{sec:4} the limit semigroup  acts in the space of continuous functions on the union of the unit circle and an annulus.} \label{pstanow2}
\end{figure}
\end{center} 

\vspace{-0.3cm}

\subsection{Reduction to analysis in a reference space}\label{reductionp}
As in Section \ref{reduction}, we use a single reference space for all the semigroups $\sem{\mc B_R},$ $ R>1.$ This time we take
\[ C(\mc R) = C(V) \otimes_\eps \CP, \]
where \[ V= [r,1-] \cup [1+, 2];\]  recall that $r\in (0,1)$ is now fixed and treated as a parameter. Each member $v$ of $C(\mc R)= C(\mc R^- \cup \mc R^+)$ is a continuous function  on either of \[ \mc R^-\coloneqq [r,1-] \times [0,2\pi] \mquad{ and } \mc R^+ \coloneqq [1+,2]\times [0,2\pi],\] separately, such that $v(\varrho,2\pi)=v(\varrho, 0)$ for all $\varrho \in V$. For any $R>1$, this space is isometrically isomorphic to $\spacec$, and a sample isomorphism is \[J_{R}: C(\mc R) \to \spacec\] given by  \[ J_R u(\rho, \phi) = u(\varrho_R(\rho), \phi ), \qquad  u \in C(\mc R), (\rho, \phi) \in \mc R_{r,R},\]
where $\varrho_R (\cdot) $ maps $V_{r,R}$ onto $V$ as follows
\[ \varrho _R(\rho) =\begin{cases} 
\frac {\rho - 1}{R-1}+1, &\rho \in [1+,R], \\
\rho ,&\rho \in [r,1-].  \end{cases} \]
The inverse is given by \[J_{R}^{-1} v(\varrho, \phi) = v(\rho_R(\varrho), \phi)
\qquad  v \in \spacec, (\varrho, \phi) \in \mc R,\]
where $\rho_R (\cdot) $ maps $V$ onto $V_{r,R}$ via
\[ \rho_R(\varrho) =\begin{cases}  1+ (R-1)(\varrho - 1), &\varrho \in [1+,2],\\\varrho , &\varrho \in [r,1-]. \end{cases} \]
Thus, we are lead to studying 
\[ \graR J_R^{-1} \e^{t\overline{\mc B_R}} J_R .\]

\begin{thm} \label{gen2} Fix $R>1$ and let $\bRj$ be  the isomorphic image of $\mc B_R $ in $C(\mc R)$ via $J_R$.  Then, $\dom{\bRj}$ is the set of functions $u \in C(\mc R)$ that are twice continuously differentiable in either of the rectangles $\mc R^+$ and $\mc R^-$, separately, and satisfy 
\begin{align}
\frac{\partial u}{\partial \varrho } (1+,\phi) & = (R-1)^2\alpha [u(1+,\phi) - u(1-,\phi)],\label{war-un-ki-bis}\\
\frac{\partial u}{\partial \varrho } (1-,\phi) & = \beta [u(1+,\phi) - u(1-,\phi)],  \nonumber 
\\ \frac{\partial u}{\partial \varrho } (r,\phi)&= \frac{\partial u}{\partial \varrho } (2,\phi)=0, \nonumber \qquad \phi \in [0,2\pi]. \end{align}
Moreover, for $u\in \dom{\bRj}$,  
\[ \bRj u(\varrho, \phi)=
\begin{cases} 
\frac {1}{(R-1)^2}  \frac{\partial^2 u}{\partial \varrho^2 }(\varrho, \phi)  +  \frac {1}{\rho_R(\varrho) (R-1)}   \frac{\partial u}{\partial \varrho } (\varrho, \phi) +  \frac 1 {(\rho_R (\varrho))^2} \frac{\partial^2 u}{\partial \phi^2 } (\varrho, \phi),  & \\
 \kappa \left [ \frac{\partial^2 v}{\partial \varrho^2 }(\varrho, \phi)  + \frac 1 \varrho \frac{\partial v}{\partial \varrho } (\varrho, \phi) +  \frac 1 {\varrho^2} \frac{\partial^2 v}{\partial \phi^2 } (\varrho, \phi)\right ], & 
\end{cases}
\]
depending on whether $(\varrho, \phi )\in \mc R^+ $ (the upper line) or $(\varrho, \phi )\in \mc R^-$ (the lower line). 
\end{thm}
\begin{proof} The proof is quite the same as the proof of Theorem \ref{gen1}, and thus we omit it.
\end{proof}
As in Section \ref{reduction},  combining Theorems \ref{mastergen}  and \ref{gen2} we  conclude that each $\bRj$ is closable and its closure is a Feller generator in $C(\mc R)$, and that the semigroup generated by $\overline{\bRj}$ is $\left (J_R^{-1} \e^{\overline {\mc B_R}t} J_R, t\ge 0\right )_{t\ge 0}.$ Therefore, our task reduces to studying
 \[ \graR  \e^{t\overline{\bRj}}\]
in $C(\mc R)$. 
\subsection{The underlying semigroup $\sem{\mc Q}$ and its properties}\label{wlasnosciqp}

In this section, as in Section \ref{wlasnosciq}, we prepare the ground for analysis based on Kurtz's Theorem. Let $Q$ be the operator in $C(V)$ with the domain composed of functions $f$ which when restricted to $[1+,2]$ are twice continuously differentiable and satisfy $f'(1)=f'(2)=0$, and let $Qf(\varrho)=f''(\varrho)$ for $\varrho \in [1+,2]$ and zero otherwise. This operator is the generator of the Feller semigroup governing the following process: in $[1+,2]$ the process is a Brownian motion with reflecting barriers at $1+$ and $2$; the process starting at $\varrho \in [r,1-]$ stays in this point for ever. From this description it is clear that 
\begin{equation}\label{asyQp} \grat \e^{tQ}f = Pf, \qquad f \in C(V) \end{equation} 
where $Pf$ equals $f$ or $\int_1^2f $ on $[r,1-]$ and $[1+,2]$, respectively.  

The semigroup \(\left (\e^{t{Q}} \otimes_\varepsilon I \right )_{t \ge 0}= \semt{Q}{0}\) generated by 
\[\mc Q = Q\otimes I\]
in $C(\mc R)$ is a version of $\sem{Q}$ in this space, and inherits its basic properties. In particular, we have  
 \begin{equation*}
 \grat \e^{t\mc Q}u = \mc P u, \qquad u \in C(\mc R) \end{equation*} 
where \[ \mc P u (\varrho, \phi ) = \begin{cases} 
  \int_1^2 u(\varrho', \phi) \ud \varrho', &(\varrho, \phi) \in \mc R^+,\\
u(\varrho, \phi), &(\varrho, \phi) \in \mc R^-. \end{cases}\]

The operator $\mc Q$ is related to $\bRj, R>1$ in the following way: 
The set of $u$ of the form $u\coloneqq f \otimes g$ where 
 $f\in \dom{Q}$ and $g \in \CP$ is contained in $\dom{\mc Q}$ and is a core for $\mc Q$. Moreover, if $u$ is of this form then  for $u_R\coloneqq f_R\otimes g$, where 
\[ f_R  (\varrho) = \begin{cases}f(\varrho)- ((f(1+) - f(1-))
\frac {(R-1)^2\alpha}\pi (2-\varrho) \sin \pi \varrho , & \varrho \in [1+,2],\\
h_R(\varrho) ,& \varrho \in [r,1-],
\end{cases} \]
and $h_R \in C^2[r,1-]$ is such that $h_R(1-) = f(1-),$ $h_R'(1-)=\beta (f(1+) - f(1-))$ and $\graR h_R = f$, we have (comp. Remark \ref{onchi})
\begin{equation} \label{zwzqp} u_R \in \dom{\bRj}, \graR u_R = u \mquad { and }  \graR (R-1)^2\bRj u_R = \mc Q u.\end{equation}

\subsection{The space $\x$ and the limit semigroup}\label{spacexp}
The space $\x$, equal to the range of the projection $\mc P$ defined in Section \ref{wlasnosciqp},
is easily seen to be the set of functions which, when restricted to $\mc R^+$, do not depend on $\varrho.$ Hence, it is convenient to identify  a member $u $ of $\x$ with a pair $\binom{g^+}{u^-}$ where $g^+ \in \CP$ is a function on the unit circle, and $u^-$ is a function on $\mc R^-.$  

Before continuing, we note that $C(\mc R)$ may be regarded as a Cartesian product of $C(\mc R^+)$ and $C(\mc R^-)$, and that $C(\mc R^+)$ in turn may be seen as the injective tensor product of $C[1+,2]$ and $C_p[0,2\pi]$. Moreover, for $u^-\in C(\mc R-)$, the function $\phi \mapsto u^-(1-,\phi)$ is a member of $C_p[0,2\pi]$; we denote it $u^-(1-,\cdot)$.

With these preparations, we are ready to define the operator $\mc O$. This time we take 
\[ \mc O \binom{g^+}{u^-} = \binom{\LB \, g^+-\psi'' \otimes (g^+-u^-(1-,\cdot))}{\kappa \Delta u^-}\]
where $\Delta$ is the Laplace operator \eqref{laplace} on $\mc R^-$ and $\psi\in C[1+,2]$ is given by $\psi(\varrho)= \frac {\alpha}{\pi} (2-\varrho)\sin \pi \varrho$; a pair $\binom{g^+}{u^-}$ belongs to the domain $\dom{\mc Q}$ iff $g^+ \in \dom{\LB}$ and  $u^-$  is a twice continuously  differentiable function on $\mc R^-$ related to $g^+$ as follows 
\begin{align}
\label{war-un-ki-}
\frac{\partial u^-}{\partial \varrho } (1-,\phi) = \beta [g^+(\phi) - u^-(1-,\phi)], \quad  \frac{\partial u^-}{\partial \varrho } (r,\phi)=0, \quad \phi \in [0,2\pi]. \end{align}  
Since $\psi'(2)=0$ and $\psi'(1+)=-\alpha$, we have $\int_1^2 \psi'' = \alpha$. Therefore,
\begin{equation}\label{POp} \mc {PO}  \binom{g^+}{u^-} = \binom{\LB \, g^+}{\kappa \Delta u^-}
 + \binom{\alpha u^-(1-,\cdot) -\alpha g^+}0.\end{equation}

\begin{thm}\label{main2} The operator $\mc {PO}$ is closable, and its closure generates a Feller semigroup in $\x$, seen as the space of functions on the union of the unit circle and $\mc R^-.$  Moreover, 
\begin{equation*}\graR \e^{t\overline{\bRj}} u = \e^{t\overline {\mc {PO}}} \mc P u , \qquad t>0, u \in C(\mc R),\end{equation*} 
and the limit is uniform in $t$ in compact subsets of $(0,\infty)$. For $u \in \x$ the relation is true also for $t=0$ and the limit is uniform in $t$ in compact subsets of $[0,\infty)$.
\end{thm} 

\begin{proof}Given $\binom{g^+}{u^-} \in \dom{O}$, we define (for the $\psi$ introduced above)
\[ u_R= \binom {g^+ + (R-1)^2 \psi \otimes [g^+ - u^-(1-,\cdot)]}{u^-}.\]
Then $u_R \in \dom{\bRj}$ and 
\[ \graR u_R = \binom{g^+}{u^-}  \mquad { while } \graR \bRj u_R = \mc O \binom{g^+}{u^-} .\]
This means that condition (a) of Kurt's Theorem is satisfied (with obvious changes amounting to replacing $r$ by $R$). 

However, in this case, condition  (b) does not hold, since $\mc{PO}$ itself is not a generator; $\mc{PO}$ is closable, though, and its closure is a generator. Fortunately, Kurtz's Theorem applies to this situation also: If  $\mc {PO}$ is closable and its closure is a generator, then the thesis still holds, and the limit semigroup is $\sem{\overline{\mc {PO}}},$ as in our theorem. Thus, we are left with showing the first statement of our theorem.

Let $\mathfrak A$ be the generator of the Feller semigroup $\semt{A^I}{\LB}$ for $A^I=A^I_{\alpha,\beta,\kappa, r}$ defined in Proposition \ref{prop2}. For $u=f\otimes g$ where $f\in \dom{A^I}$ and $g\in \dom{\LB}$, 
\[ \mathfrak A u(\varrho,\phi) = [\kappa \varrho^2 f''(\varrho) + \kappa \varrho f'(\varrho)]g(\phi) + f(\varrho) g''(\phi) =\kappa \varrho^2 \Delta u(\varrho, \phi),\, (\varrho,\phi) \in \mc R^-, \]
and
\begin{align*} \mathfrak A u(1+,\phi) &= [-\alpha f(1+) +\alpha f(1-)] g(\phi) + f(1+) g''(\phi)\\&= \LB u(1+,\phi) - \alpha g^+(\phi) + \alpha u^-(1-,\phi), \qquad \phi \in [0,2\pi].\end{align*}
This means that $\mathfrak A u = \chi^2 \mc Au$ where $\chi \in C(V_r)$ is given by $\chi (\varrho)=\varrho, \varrho \in V_r.$ This relation may be extended to all linear combinations $\mc D$ of simple tensors with factors in $\dom{A^I}$ and $ \dom{\LB}$:
\[ \mathfrak A u = \chi^2 \mc Au, \qquad u \in \mc D. \]
This together with the fact that $\mathfrak A$ is a conservative Feller generator implies that that $\mc A$ is closable and its closure is a conservative Feller generator. We omit the details of this reasoning, since it is quite analogous to that presented in the proof of Theorem \ref{mastergen} (see the lines following formula \eqref{zwiazek}).
 \end{proof}

\section{A PDE coupled with an ODE as a master limit equation: convergence as $R\to 1$ and $\kappa \to \infty$}\label{sec:5}

In the last scenario of interest, a thin outer annulus communicates with a thick inner annulus where, however, diffusion is very fast (see  Figure \ref{pstanow3}). Hence, we consider the semigroups of Theorem \ref{mastergen} in the case where $R\to 1$ while $\kappa \to \infty$. If the limit is to be non-trivial, the rate at which $R\to 1$ must be comparable to the rate at which $\kappa \to \infty,$ and therefore we assume that there is a $\gamma >0$ such that 
\begin{equation}\label{kandR} \kappa (R-1)^2=\gamma. \end{equation}
We also need to rescale transmission conditions appropriately; this leads us to considering 
\[ \mc B_\kappa \coloneqq\mc A_{\alpha \sqrt{\gamma/ \kappa} ,\beta\kappa^{-1},\kappa,r,1+\sqrt{\gamma/\kappa}}.\]

\begin{center}
\begin{figure}
\includegraphics[scale=1.2]{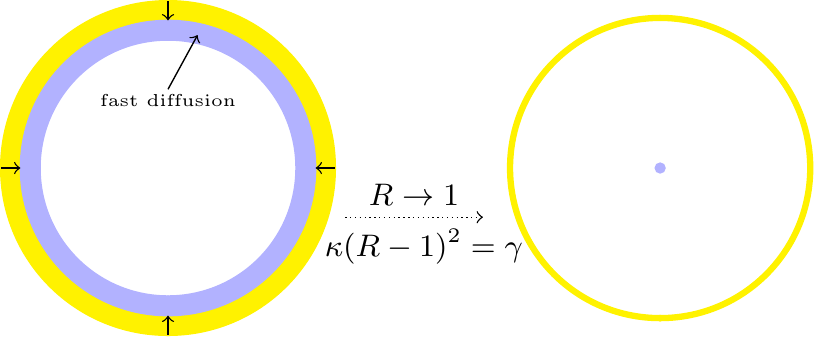}
 \caption{In Section \ref{sec:5} the limit semigroup  acts in the space of continuous functions on the union of the unit circle and a point.} \label{pstanow3}
\end{figure}
\end{center} 
\vspace{-0.3cm}

\subsection{Reduction to analysis in a reference space}\label{reductionb}
For a reference space, where isomorphic images of $\sem{\overline {\mc B_\kappa}}$ are defined, we take the space $C(\mc R)$ from Section \ref{reductionp}. The isomorphisms defined in the latter section work here as well, but we need to take into account the fact that now $R$ is seen as a function of $\kappa$. Hence, we consider the isomorphism $J_\kappa$ equal to the old $J_{R}$ with $R=1+ \sqrt{\gamma/\kappa}$, and are interested in the limit 
\[ \grak J_\kappa^{-1} \e^{\overline{\mc B_\kappa}} J_\kappa.\]
Theorem \ref{gen2} takes now the following form.

\begin{thm} \label{gen3} Fix $\kappa>0$ and let $\mc C_\kappa $ be  the isomorphic image of $\mc B_\kappa $ in $C(\mc R)$ via $J_\kappa$.  Then, $\dom{\mc C_\kappa}$ is the set of functions $u \in C(\mc R)$ that are twice continuously differentiable in either of the rectangles $\mc R^+$ and $\mc R^-$, separately, and satisfy 
\begin{align}
\frac{\partial u}{\partial \varrho } (1+,\phi) & = \alpha \gamma\kappa^{-1} [u(1+,\phi) - u(1-,\phi)],\label{war-un-ki-bis-bis}\\
\frac{\partial u}{\partial \varrho } (1-,\phi) & = \beta \kappa^{-1}[u(1+,\phi) - u(1-,\phi)],  \nonumber 
\\ \frac{\partial u}{\partial \varrho } (r,\phi)&= \frac{\partial u}{\partial \varrho } (2,\phi)=0, \nonumber \qquad \phi \in [0,2\pi]. \end{align}
Moreover, for $u\in \dom{\mc C_\kappa}$,  
\[ \mc C_\kappa u(\varrho, \phi)=
\begin{cases} 
\frac {\kappa}{\gamma}  \frac{\partial^2 u}{\partial \varrho^2 }(\varrho, \phi)  + \frac {\sqrt{\kappa/\gamma} }{\rho_R(\varrho)}   \frac{\partial u}{\partial \varrho } (\varrho, \phi) +  \frac 1 {(\rho_R (\varrho))^2} \frac{\partial^2 u}{\partial \phi^2 } (\varrho, \phi),  & (\varrho, \phi )\in \mc R^+, \\
 \kappa \left [ \frac{\partial^2 v}{\partial \varrho^2 }(\varrho, \phi)  + \frac 1 \varrho \frac{\partial v}{\partial \varrho } (\varrho, \phi) +  \frac 1 {\varrho^2} \frac{\partial^2 v}{\partial \phi^2 } (\varrho, \phi)\right ], & (\varrho, \phi )\in \mc R^-, 
\end{cases}
\]
 where 
\[ \rho_\kappa(\varrho) = 1+ \sqrt{\gamma/\kappa}(\varrho - 1), \qquad \varrho \in [1+,2].  \] 
\end{thm}
This theorem, when combined with Theorem \ref{mastergen}, shows that $\mc C_\kappa$ is closable and its closure generates the isomorphic image of $\sem{\overline{\mc B_\kappa}}$ in $C(\mc R)$; we are interested in the limit $\grak \e^{t\overline {\mc C_\kappa}}$. 

\subsection{The underlying semigroup $\sem{\mc Q}$ and its properties}\label{wlasnosciqb}

Let $Q$ be the operator in $C[1+,2]$ with domain equal to the set of twice continuously differentiable functions $f$ such that $f'(1+)=f'(2)=0$, given by $Qf=\gamma^{-1} f''.$ This operator is known to generate the Feller semigroup governing a Brownian motion on $[1+,2]$ with diffusion coefficient $\gamma^{-1}$ and reflecting barriers at both interval ends. We also have 
\begin{equation*}
\grat \e^{tQ}f = \mathsf Pf, \qquad f \in C[1+,2] \end{equation*} 
where $\mathsf Pf$ equals $\int_1^2f $.  The semigroup $\semt{Q}{0}$ generated by 
$ Q = Q\otimes I$
(where $I=I_{C_p[0,2\pi]}$) in $C(\mc R^+)$ is a version of $\sem{Q}$ in this space, and inherits its basic properties. In particular, 
\begin{equation*}
 \grat \e^{t(Q\otimes I)}u^+= P^+u^+, \qquad u^+ \in C(\mc R^+) \end{equation*} 
where $P^+u^+(\varrho,\phi)=\int_1^2u^+(\varrho',\phi)\ud \varrho'$.

Next, let $\Delta_N$ be the Laplace operator \eqref{laplace} in $C(\mc R^-)$, defined on the set of twice continuously differentiable functions with normal derivatives vanishing at the boundary of $\mc R^-$. This operator is closable and its closure generates the Feller semigroup in $C(\mc R^-)$ describing Brownian motion  in the inner annulus, with reflecting barriers at the circles with radiuses $r$ and $1$, and we have  
\begin{equation*}
\grat \e^{t\overline{\Delta_N}}u^- = P^-u^-, \qquad u^- \in C(\mc R^-) \end{equation*} 
where $P^-u^-\coloneqq \frac 1{area(\mc R^-)}\int_{\mc R^-}u^-=\frac 1{2\pi (1-r)}\int_{\mc R^-}u^- $. 

Therefore, treating the space $C(\mc R)$ as the Cartesian product of $C(\mc R^+)$ and $C(\mc R^-)$, we see that 
the operator $\mc Q$ in $C(\mc R)$ with domain $\dom{Q\otimes I}\times \dom{\overline{\Delta_N}}$ mapping $\binom{u^+}{u^-} $ to $\binom{(Q\otimes I)u^+}{\overline {\Delta_N}u^-}$ is a Feller generator in $C(\mc R)$.  The related process is a reflected Brownian motion in the lower annulus; a particle starting in the upper annulus performs a Brownian motion merely in the radial direction.  
As a result we have  
 \begin{equation*}
\grat  \e^{t\mc Q}  u = \mc P u, \qquad u \in C(\mc R) \end{equation*} 
where $ \mc P (u^+,u^-) = (P^+ u^+, P^- u^-). $

To establish connection between $\mc Q$ and the operators $\mc C_\kappa$, we note that the set of $u = \binom {u^+}{u^-}\in \dom{\mc Q}$ where $u^+ = f\otimes g$ with $f \in \dom{Q}, g \in C_p[0,2\pi]$ and $u^-\in \dom{\Delta_N}$ is a core for $\mc Q$. For a $u$ of this form we define 
\begin{equation}\label{ukappa} u_\kappa = u - \kappa^{-1} \psi \otimes [f(1+)g - u(1-,\cdot)]\end{equation}
where $\psi \in C([r,1-] \cup [1+,2])$ is defined as follows 
\[ \psi (\varrho) = \begin{cases} \frac {\alpha\gamma}{\pi} (2-\varrho) \sin \pi \varrho, & \varrho \in [1+,2], \\
 \frac {\beta (1-r)(\varrho-r)}{\pi} \sin 
 \pi \frac{\varrho-r}{1-r}, &\varrho \in [r,1-];\end{cases}\]
in particular $\psi (1+)=\psi(1-)=0$ and $\psi'(1+)= -\alpha \gamma, \psi' (1-)=-\beta .$ Then
\begin{equation} \label{zwzqb} u_\kappa \in \dom{\mc C_\kappa}, \grak u_\kappa = u \mquad { and }  \grak \kappa^{-1} \mc C_\kappa u_\kappa = \mc Q u.\end{equation}

\newcommand{\ixi}{k^-}

\subsection{The space $\x$ and the limit semigroup}\label{spacexbis}
The space $\x$ equal to the range of the operator $\mc P$ defined in Section \ref{wlasnosciqb} is composed of functions which when restricted to $\mc R^+$ do not depend on $\varrho$ and when restricted to $\mc R^-$ depend neither on $\varrho$ nor on $\phi.$ It will be convenient to identify a $u\in \x$  with a pair $\binom{g^+}{\ixi}$ where $g^+$ is a function in $ C_p[0,2\pi]$ and $\ixi$ is a real number.

Let $\mc O: \dom{\mc O} \to C(\mc R)$ with $\dom{O}\subset \x$ equal to the set of pairs $\binom{g^+}\ixi $ with $g^+\in \dom{\LB}$, be given by 
\[ \mc O \binom{g^+}\ixi = \binom{\LB g^+}\ixi - \gamma^{-1} \psi''\otimes (g^+ - k^-),\] 
where $\psi$ has been defined above.  We note that since $\int_1^2 \psi = \alpha \gamma $ and $\int_r^1 \psi = -\beta ,$
\begin{equation}\label{operatorpo} \mc {PO} \binom{g^+}\ixi = \binom{\LB g^+}0  + \binom{ \alpha k^- - \alpha g^+}{\gamma^{-1}\beta \int_0^{2\pi}g^+ - \gamma^{-1}\beta \ixi}.\end{equation}

\begin{thm}The operator $\mc{PO}$ is a generator, and we have 
\begin{equation*}\grak \e^{t\overline{C_\kappa}} u = \e^{t\mc {PO}} \mc P u , \qquad t>0, u \in C(\mc R),\end{equation*} 
and the limit is uniform in $t$ in compact subsets of $(0,\infty)$. For $u \in \x$ the relation is true also for $t=0$ and the limit is uniform in $t$ in compact subsets of $[0,\infty)$.
\label{main3} \end{thm}

\begin{proof} Vectors of the form $\binom{g^+}\ixi \in \dom{\mc O}$ are members of $\dom{\mc Q}$ as well, and therefore $u_\kappa$ defined in \eqref{ukappa} belongs to $\dom{\mc C_\kappa}$. Moreover, 
\[ \grak \mc C_\kappa u_\kappa = \mc O u. \] 
(and, as before, $\grak u_\kappa = u$). By Kurtz's Theorem, therefore, we are left with showing that $\mc {PO}$ is a generator. However, $\binom{g^+}\ixi \mapsto \binom{\LB g^+}0$ is obviously a generator, and $\mc {PO}$ is its bounded perturbation. Hence, we are done by the Phillips Perturbation Theorem.  \end{proof}

\section{Concluding remarks}\label{cr} 
Let us look more closely at the processes and equations obtained in the limit. In the case of Theorem \ref{main1}, both layers become thiner and thiner, and in the limit the state-space of the process involved is composed of the upper and lower sides of the circle (the membrane); see Figure \ref{fig4} (a). The Kolmogorov backward equation for the process (which differs from \eqref{prawieconvex} by constants accounting for different diffusion coefficients in the upper and the lower side, and for rescaling of jump intensities, see Remark \ref{parametry}), is thus a system of two partial differential equations on the unit circle describing diffusion on both sides of the membrane: 
\begin{equation}
\frac {\ud } {\ud  t} \binom {u_+}{u_-} = \begin{pmatrix}\LB & 0 \\ 0 & \kappa \LB \end{pmatrix}  \binom {u_+}{u_-}  + \begin{pmatrix}- \alpha \gamma^{-1} & \alpha\gamma^{-1}\\  \kappa \beta  & -\kappa \beta\end{pmatrix} \binom {u_+}{u_-} \label{prawieconvexprim}. \end{equation}
These PDEs are coupled by the second term, describing jumps from the upper to the lower side, and vice versa. 

Theorem \ref{main2} describes a different situation (see Figure \ref{fig4} (b)). Here, only the thickness of the upper layer converges to zero so that in the limit this layer may be identified with the upper side of the membrane. The state-space, therefore, is the union of the upper side of the circle and the lower annulus, whose thickness is fixed. As a result, the Kolmogorov equation is a system of two partial differential equations, one of which is one-dimensional and the other is two-dimensional:
 \begin{align}\label{Kol2}  {\frac{\partial g^+}{\partial t}}&= \LB \, g^+ + \alpha u^-(1-,\cdot) -\alpha g^+,\nonumber\\
 \frac {\partial u^-}{\partial t} &= \kappa \Delta u^-.\end{align} 
These equations are coupled in an interesting way: the term  $ \alpha u^-(1-,\cdot) -\alpha g^+$ from the first line describes jumps from the upper side of the circle to the inside of the annulus. However, particles diffusing in the annulus do not reach the upper side of the circle by jumps but rather by filtering through the membrane, and this is described by the first condition in  \eqref{war-un-ki-}. In other words, these PDEs are coupled by both a jump term and by a transmission condition. Needless to say, both the term and the transmission condition are residues of the transmission conditions featuring in the approximating equations. 

Finally, in Theorem \ref{main3}, the thickness of the upper layer diminishes and at the same time diffusion in the lower layer becomes faster and faster (see Figure \ref{fig4} (c)). As a result, elements of the lower layer communicate so quickly that they become indistinguishable, and in the limit all points of the lower layer are lumped into one new state. Thus, the state-space of the limit process is the union of the upper side of the circle and this new combined state, and the Kolmogorov equation is a pair of  differential equations, one of which is partial and the other is ordinary: 
 \begin{align}\label{Kol3}  \frac{\partial g^+}{\partial t}&=
\LB g^+  + \alpha \ixi  - \alpha g^+ ,\nonumber \\
\frac{\ud \ixi}{\ud t}& = 
\gamma^{-1}\beta \int_0^{2\pi}g^+ - \gamma^{-1}\beta \ixi.\end{align}
The first of these describes a Brownian motion on the (outer part of the) unit circle, accompanied by jumps from the circle to the additional point. At this point the process stays for an exponential time with parameter $\gamma^{-1}\beta$ and then jumps back to the circle, and the distribution of its position after such a jump is uniform across the circle. In contrast to Theorem \ref{main2}, no transmission conditions are needed here.

These scenarios exemplify  variety of manners transmission conditions become, in the thin layer approximation, integral parts of the limit master equation. In particular, it is clear from these examples that transmission conditions and the terms describing jumps play complementary roles. Recalling interpretation of transmission conditions in terms of the L\'evy local time a Brownian traveller spends at the boundary, we see that while transmission conditions speak of exponential epochs with respect to this L\'evy local time before a particle fiters through the membrane, the jump terms speak of exponential epochs before jumps in `regular' time.

A look at \eqref{prawieconvexprim} shows, however, that in applying the principle `transmission conditions in the thin layer approximation become integral parts of the limit equation', one must be  
careful, and needs to consider all factors influencing the shape of the limit equation. For, in the particular case under consideration, the `expected' form of transmission conditions (cf. \eqref{prawieconvex}) is altered by the  differences in thickness of the layers involved. 
In extreme cases, like that of Theorem \ref{main2}, where one layer is `infinitely thicker' than the other, perhaps only some transmission conditions `jump into' the master limit equation, and some other are transferred from the approximating equations in an unchanged form. Finally, as in Theorem \ref{main3}, when modeling a particular phenomenon, a thin layer might not be the sole circumstance that needs to be taken into account: for instance, a thin layer may be accompanied by a relatively fast diffusion, and then the limit master equation may be influenced by the transmission conditions in a yet another way.

\section{Appendix}\label{appendix}

\subsection{An auxiliary generation result needed for Theorem \ref{mastergen}}

Let $a<0$ and $b>0$ be real numbers, let $C(U)$ be the space of continuous functions on 
\begin{equation}\label{u} U= U_{a,b} \coloneqq [a,0-]\cup [0+,b].\end{equation}  As in the main text, $0-$ and $0+$ are distinct points to the immediate left and to the immediate right of an imaginary membrane at $0$. Thus, an $f \in C(U)$ is a function on the interval $[a,b]$ which is continuous in this interval save perhaps at $x=0$ where it has left and right limits (identified with $f(0-)$ and $f(0+)$). 

Given parameters  $\alpha, \beta\ge 0 $ and $\kappa >0$, let \[A=A_{\alpha,\beta,\kappa,a,b}\] be the operator in $C(U)$ defined by \[ Af (x) = \begin{cases}  f''(x), & x\in [0+,b], \\ \kappa f''(x), & x\in [a,0-] \end{cases} \]  (with right-hand and left-hand derivatives at appropriate intervals' ends); its domain $\dom{A}$ is the set of functions which, when restricted to either subinterval, are of class $C^2$ and satisfy the following boundary and transmission conditions 
\begin{align}
\nonumber  f'(a) &= f'(b) = 0 ,\\
\nonumber f'(0+) &= \alpha [f(0+) - f(0-)],\\
f'(0-) &= \beta [f(0+) - f(0-)] \label{bandt} .\end{align}
The following lemma is a special case of Proposition 2.1 in \cite{nagrafach}, but we provide its more elementary and more direct proof here. 

\begin{lem}\label{lem1} The operator $A= A_{\alpha,\beta,\kappa, r, R}$ is the generator of a conservative Feller semigroup in $C(U)$.
\end{lem}

\begin{proof}\ \

(a) It is clear that $\dom{A}$ is dense in $C(U).$ Moreover, $A$ satisfies the maximum principle: if $f \in \dom{A}$ and for some $x_0 \in U$, $f(x_0)=\max_{x \in U} f(x)$, then $Af (x_0)\le 0$. The latter principle is clear for $x_0$ in the interior of $U$. Also, if the maximum of $f$ is attained at the interval's end, and the first derivative vanishes there, then so does the second derivative; this takes care of the cases $x_0=a,b$. Finally, if $x_0=0-$, then $f'(0-)\ge 0$ and, on the other hand, by the second condition in \eqref{bandt}, $f'(0-)\le 0$, implying $f'(0-)=0$. As above, it follows that $f''(0-)=0$, proving the principle in this case also. The case $x_0=0+$ is analogous.  

(b) We need to show the range condition is satisfied (see e.g. \cite[Lemma 2.1, p. 165]{ethier} or \cite[Section 8.3.4]{kniga}): for $g\in C(U)$ and $\lam >0$ there is an $f \in \dom{A}$ such that $\lam f - Af =g $. To this end, we search for $f $ of the form 
\[
f(x) = \begin{cases} 
 C_2 \cosh \sqrt \lam (b-x) +h(b) \sinh \sqrt \lam (b-x) + h(x), & x \in [0+,b],\\
 C_1 \cosh \sqrt {\lam/\kappa} (x-a) - h(a) \sinh \sqrt{\lam/\kappa} (x-a) + h(x), & x \in [a,0-],
  \end{cases}
\]
where $C_1$ and $C_2$ are unknown constants and 
\[
h(x) = \begin{cases} \frac 1{2\sqrt \lam} \int_0^b \e^{-\sqrt \lam |x-y|}g(y)\ud y, & x \in [0+,b],
\\
\frac 1{2\sqrt {\kappa \lam}} \int_a^0 \e^{-\sqrt {\lam/\kappa} |x-y|}g(y)\ud y, & x \in [a,0-].
 \end{cases}\]
It may be checked that such an $f$ is twice continuously differentiable with $\kappa f'' = \lam f - g$ on $[a,0-]$ and $f'' = \lam f - g$ on $[0+,b]$, and
$f'(a)=f'(b)=0$ (use $h'(a)= \sqrt {\lam/\kappa} h(a)$ and $h'(b)=-\sqrt \lam h(b)$). Thus, our task reduces to showing that for some $C_1$ and $C_2$, the second and third conditions in \eqref{bandt} are satisfied. These conditions, however, can be written as a system of two linear equations for $C_1$ and $C_2$ with the main determinant  
\[ 
\begin{vmatrix}
\sqrt {\lam/\kappa} \sinh \sqrt {\lam/\kappa} (-a) + \beta \cosh \sqrt {\lam/\kappa} a & -\beta \cosh \sqrt \lam b \\
-\alpha \cosh \sqrt {\lam/\kappa}a& \sqrt \lam \sinh \sqrt \lam b + \alpha \cosh \sqrt \lam b
\end{vmatrix}
\]
larger than $0$.

(c) The semigroup generated by $A$ is conservative since the constant function $1_U$ belongs to $\dom {A}$ and $A1_U=0$.   \end{proof}

Next, given $r \in (0,1)$ and $R>0$ we consider the union
\[ V_{r,R}= [r,1-] \cup [1+, R] \]
where $1-$ and $1+$ are positions, respectively, to the immediate left, and to the immediate right, of $1$, where the membrane is now located.  As in the case of $C(U)$, the space $C(V_{r,R})$ of continuous functions on $V_{r,R}$ is composed of functions  $f$ defined on the interval $[r,R]$ which are continuous in this interval save perhaps at $x=1$ where they have left and right limits.
This space is isometrically isomorphic to $C(U_{\ln r,\ln R})$.
The isomorphism we have in mind is $I: C(U) \to C(V_{r,R})$ given by 
\begin{equation}\label{ii} 
Ig (\rho) = g(\ln \rho ), \qquad g \in C(U), \rho \in V_{r,R}, \end{equation}
and its inverse is 
\[ I^{-1} f  (x) = f(\e^x), \qquad f \in C(V_{r,R}), x \in U .\]
 
\begin{lem}\label{lem2} For $r\in (0,1)$ and $R>1$ let $a=\ln r$ and $b=\ln R.$ Also, let $\chi \in C(V_{r,R})$ be given by $\chi (\rho) = \rho, \rho \in V_{r,R}.$ The isometrically isomorphic image, say $A^I(=A^I_{\alpha,\beta,\kappa,r,R})$ of the operator $A=A_{\alpha,\beta,\kappa,r,R}$ of Lemma \ref{lem1} via the isomorphism $I$ of \eqref{ii} is given by 
\begin{equation}\label{ai} A^I f =\begin{cases} \chi^2 f'' + \chi f' & \text{on }[1+,R],\\ \kappa (\chi^2 f'' + \chi f') & \text{on }[r,1-]\end{cases} \end{equation}
and its domain $\dom{A^I}$ is the set of  twice continuously differentiable functions on $V_{r,R}$ such that 
\begin{align}
\nonumber  f'(r) &= f'(R) = 0 ,\\
\nonumber f'(1+) &= \alpha [f(1+) - f(1-)],\\
\label{bandtdlar} f'(1-) &= \beta [f(1+) - f(1-)].\end{align}
\end{lem}

\begin{proof} An $f \in C(V_{r,R})$ belongs to $\dom{A^I}$ iff $I^{-1}f $ belongs to $\dom{A}$ and then $A^I f = IAI^{-1}f$ (see e.g. \cite[Section 7.4.22]{kniga}). By the definition of $I$ and $I^{-1}$ it is clear that $f$ is twice continuously differentiable iff so is $I^{-1}f$.  Moreover, 
\[ AI^{-1} f (x) = \begin{cases}  \frac {\ud^2}{\ud x^2} f(\e^x) = \e^{2x} f''(\e^x) + \e^x f'(\e^x), & x \in [0+,b] ,\\ \kappa \frac {\ud^2}{\ud x^2} f(\e^x) = \kappa [\e^{2x} f''(\e^x) + \e^x f'(\e^x)], & x \in [a,0-] ,\end{cases}\]
showing that for $f \in \dom{A^I}$, $A^I$ is given by  \eqref{ai}.

Next, an $f \in \dom{A^I}$ is necessarily of the form $f=Ig$ where $g$ belongs to $\dom{A}$ and thus in particular satisfies \eqref{bandt} with $a$ and $b$ replaced by $\ln r$ and $\ln R$, respectively. Since $f'(\rho) = \frac 1 \rho g'(\ln \rho), \rho \in V_{r,R},$ we have $f'(r) = \frac 1r g'(\ln r) =0$ and similarly $f'(R)=0$.  By the same token, $f'(1-)=g'(0-)$ and $f'(1+)=g'(0^+)$. Since $f(1+)=g(0+)$ and $f(1-)=g(1-),$ relation \eqref{bandt} for $g$ implies \eqref{bandtdlar} for $f$. We have shown that $\dom{A^I}$ is contained in the set of twice continuously differentiable functions satisfying \eqref{bandtdlar}; the opposite inclusion is proved analogously.  
\end{proof}

We note that the isomorphism $I$ preserves the lattice structures of the spaces involved and maps $1_{U_{\ln r,\ln R}}$ to $1_{V_{r,R}}$. This establishes the following result. 

\begin{prop}\label{prop1} $A^I$, as an isomorphic image of a conservative Feller generator, is a conservative Feller generator.    
\end{prop}

\subsection{An auxiliary generation result needed for Theorem \ref{main2}}

Let $a<0, \kappa >0$ and non-negative $\alpha,\beta$ be given. Consider the space $C(U)$ of 
of continuous functions on the union \[ U=U_a \coloneqq [a,0-]\cup \{0+\},\] and the operator $A= A_{\alpha,\beta,\kappa,a}$ in $C(U)$ given by 
\[ Af (x) = \begin{cases} 
-\alpha f(0+) + \alpha f(0-), &x = 0+,\\
\kappa f''(x), &x \in [a,0-],\end{cases}\]
with domain composed of functions $f$ which when restricted to $[a,0-]$ are twice continuously differentiable and satisfy 
\begin{equation}\label{a-conditions} 
f'(a) =0 \mquad { and } f'(0-) = \beta (f(0+)-f(0-)).  \end{equation}

\begin{lem}\label{lem3} The operator $A$ is a conservative Feller generator.\end{lem}
\begin{proof} Consider  first the case where $\alpha=0$. It is clear that $A$ is densely defined, and arguing as in Lemma \ref{lem1} we see that it satisfies the maximum principle. Since $1_U\in \dom{A}$ and $A1_U=0$, we are left with showing that the range condition is satisfied. Similarly as in Lemma \ref{lem1}, given $g \in C(U)$ and $\lam >0$, we search for a solution $f\in \dom{A}$ of the equation $\lam f - Af=g$ among $f $ of the form 
\[ f(x) = \begin{cases} \lam^{-1}g(0+),& x = 0+,\\
C \cosh \sqrt{\lam/\kappa} (x-a) - h(a) \sinh  \sqrt{\lam/\kappa} (x-a)+ h(x), & x \in  [a,0-],
\end{cases} \]
where $h(x) = \frac 1{2\sqrt {\kappa \lam}} \int_a^0 \e^{-\sqrt {\lam/\kappa} |x-y|}g(y)\ud y, x \in [a,0-]$ and $C$ is a constant; in particular, $\lam f- \kappa f'' = g$ on $[a,0-]$. Regardless of the choice of $C$, the first condition in \eqref{a-conditions} is satisfied. Moreover, as a bit of algebra shows, the second condition is satisfied if $C \left (\sqrt{\lam/\kappa} \sinh \sqrt{\lam/\kappa}(-a)+ \beta \cosh \sqrt{\lam/\kappa}a \right ) = F_\beta (g)$ where $F_\beta$ is a certain functional on $C(U)$. Since the expression in the parentheses on the left is $>0$,  $C$ may be chosen so that $f$ belongs to $\dom{A}$. This completes the proof in the case $\alpha=0$. 

Since $ A_{\alpha,\beta,\kappa,a}$ is a bounded perturbation of $A_{0,\beta,\kappa,a}$, by the Phillips Perturbation Theorem $A_{\alpha,\beta,\kappa,a}$ is a generator, and in particular it satisfies the range condition (for sufficiently large $\lam $). Since the maximum principle is also satisfied and $1_U $ belongs to $\dom{A_{\alpha,\beta,\kappa,a}}$ with $A_{\alpha,\beta,\kappa,a}1_U=0,$ $ A_{\alpha,\beta,\kappa,a}$ generates a conservative Feller semigroup. 
\end{proof}

Let $C(V_{r})$ be the space of continuous functions on the union $[r,1-]\cup \{1+\}$. Since 
the map  $I: C(U_{\ln r} ) \to C(V_r)$ given by 
\begin{equation}\label{iip} 
Ig (\rho) = g(\ln \rho ), \qquad g \in C(U_{\ln r}), \rho \in V_{r}, \end{equation}
with the inverse 
\[ I^{-1} f  (x) = f(\e^x), \qquad f \in C(V_{r}), x \in U_{\ln r} ,\]
establishes an isometric isomorphism which preserves the lattice structures of the spaces involved, and maps $1_{U_{\ln r}}$ to $1_{V_r}$, as an immediate consequence of Lemma \ref{lem3} we obtain the following result.

\begin{prop}\label{prop2} Let $A^I=A^I_{\alpha,\beta,\kappa,r}$ be the following operator in $C(V_r)$. Its domain is composed of functions $f$ which when restricted to $[r,1-]$ are twice continuously differentiable, and satisfy the following conditions: 
\begin{equation}\label{ai-conditions} 
f'(r) =0 \mquad { and } f'(1-) = \beta (f(1+)-f(1-)), \end{equation}
 and 
 \[ A^If(\rho ) = \begin{cases} 
 -\alpha f(1+) + \alpha f(1-), & \rho = {1+}, \\
 \kappa \rho^2 f''(\rho ) +\kappa \rho f'(\rho) , & \rho \in [r,1-].\end{cases} \]
 Then $A^I$, as an isomorphic image of $A$ of Lemma \ref{lem3}, is a conservative Feller generator in $C(V_r).$\end{prop}
  
\vspace{0.2cm}
\textbf {Acknowledgment.}  This research is supported by National Science Center (Poland) grant
2017/25/B/ST1/01804.


\bibliographystyle{plain}
\bibliography{../../../../../bibliografia}
\end{document}